\RequirePackage{fix-cm}
\RequirePackage{amsmath}
\documentclass[smallextended, envcountsame, numbook]{svjour3}
\smartqed  % flush right qed marks, e.g. at end of proof

\usepackage{graphicx}
\usepackage{mathptmx}

\usepackage{amsxtra}
\usepackage{amscd}

\usepackage{amsthm}

\usepackage{amsfonts}
\usepackage{amssymb}
\usepackage{eucal}

\usepackage[all]{xy}
\usepackage{graphicx}
\usepackage{tikz}
\usepackage{tikz-cd}
\usepackage{latexsym,epsfig,verbatim}
\usepackage[margin=1.1in,letterpaper,portrait]{geometry}
\usepackage{thmtools}
\usepackage{thm-restate}

\theoremstyle{plain}
\newtheorem{thm}{Theorem}[section]
\newtheorem{prop}[thm]{Proposition}
\newtheorem{lem}[thm]{Lemma}
\newtheorem{cor}[thm]{Corollary}

\theoremstyle{definition}
\newtheorem{defn}[thm]{Definition}

\theoremstyle{remark}
\newtheorem{rmk}[thm]{Remark}
\newtheorem{conj}[thm]{Conjecture}
\newtheorem{que}[thm]{Question}
\newtheorem{eg}[thm]{Example}

\usetikzlibrary{matrix,calc, arrows,decorations.pathmorphing}

\numberwithin{equation}{section}

\newcommand\nc{\newcommand}
\nc\on{\operatorname}
\nc\renc{\renewcommand}
\newcommand\ssec{\subsection}
\newcommand\sssec{\subsubsection}
\renewcommand{\iff}{\Leftrightarrow}
\newcommand{\id}{\mathrm{id}}
\newcommand\rk{\operatorname{rk}}

\newcommand\Stab{\operatorname{Stab}}

\newcommand\Aut{\operatorname{Aut}}
\newcommand\op{\operatorname{op}}

\journalname{Order}

\begin{document}

\title{Peckness of Edge Posets}

\author{David Hemminger      \and
        Aaron Landesman 		 \and
        Zijian Yao
}

\institute{D. Hemminger \at
              Duke University, 6300 Lakeland Dr., Raleigh, NC 27612
               \email{dhemminger@math.ucla.edu}           
           \and
           A. Landesman \at
             Harvard University, 323 Eliot Mail Center, Cambridge, MA 02138
              \email{aaronlandesman@gmail.com} \and
           Z. Yao \at
           Brown University,
             7285 Brown University, Providence, RI 02912
              \email{zijian.yao.math@gmail.com}			%  \\
}

%\date{March 16, 2015}

\maketitle

\begin{abstract}
For any graded poset $P$, we define a new graded poset, $\mathcal E(P)$, whose elements are the edges in the Hasse diagram of P. For any group $G$ acting on the boolean algebra $B_n$ in a rank-preserving fashion we conjecture that $\mathcal E(B_n/G)$ is Peck. We prove that the conjecture holds for ``common cover transitive'' actions. We give some infinite families of common cover transitive actions and show that the common cover transitive actions are closed under direct and semidirect products.

\end{abstract}

\newpage

%%%%%%%%%%%%%%%%% Introduction %%%%%%%%%%%%%%%%%%
%%%%%%%%%%%%%%%%%%%%%%%%%%%%%%%%%%%%%%%%%%%%%
\section{Introduction}\label{sec:introduction}

Let $P$ be a finite graded poset of rank $n$.  In this paper we study the structure of the edges in the Hasse diagram of $P$.  To this end, we define an endofunctor $\mathcal{E}$ on the category of finite graded posets with rank-preserving morphisms as follows.

\begin{defn}
\label{defn:functor_of_edges}
For $\mathcal P$ the category of graded posets, define the {\it functor of edges} $\mathcal{E}\colon\mathcal{P} \rightarrow \mathcal{P}$ as follows. Given $P \in \mathcal P$, the elements of the graded poset $\mathcal{E}(P)$ are pairs $(x, y)$ where $x,y\in P$, $x\le_P y$, and $\rk(y) = \rk(x) + 1$. Define the covering relation $\lessdot_{\mathcal{E}}$ on $\mathcal{E}(P)$ by $(x, y) \lessdot_{\mathcal{E}} (x^\prime, y^\prime)$ if $x\lessdot_P x^\prime$ and $y\lessdot_P y^\prime$. Then define the relation $\le_{\mathcal{E}}$ on $\mathcal{E}(P)$ to be the transitive closure of $\lessdot_{\mathcal{E}}$.

Let $Q$ be a finite graded poset of rank $n$. Given a morphism $f\colon P\rightarrow Q$, define $\mathcal{E}(f)\colon \mathcal{E}(P)\rightarrow \mathcal{E}(Q)$ by $\mathcal{E}(f)(x,y) = (f(x), f(y))$.
\end{defn}

We show that $\mathcal{E}(P)$ is a well-defined graded poset in Section \ref{sec:functor_of_edges}.  Note that an edge in the Hasse diagram of $P$ can be identified with a pair $(x,y)\in P\times P$ such that $x\lessdot y$, and the edges in the Hasse diagram are in bijection with elements $(x,y)\in \mathcal E(P)$ via this identification.  With this in mind, we will frequently refer to $\mathcal{E}(P)$ as the {\it edge poset} of $P$. 

\begin{eg}
We give an example of an edge poset in Figure \ref{fig:EP_definition_example}.  Note that it is important we declare the relation $\leq_\mathcal E$ to be the transitive closure of $\lessdot_{\mathcal E}$.  If instead we defined a relation $\le_{\mathcal{E^\prime}}$ on $\mathcal{E}(P)$ by $(x, y) \leq (a, b)$ if $x \leq a, y \leq b$, then $\mathcal{E}(P)$ would not necessarily be a graded poset.  In Figure \ref{fig:EP_definition_example} it is clear that $\mathcal E(P)$ is a graded poset under relation $\le_{\mathcal{E}}$, with $\rk(x,y) = \rk(x)$, but the Hasse diagram on the right represents a poset which does not have a grading under the relation $\le_{\mathcal{E^\prime}}$.

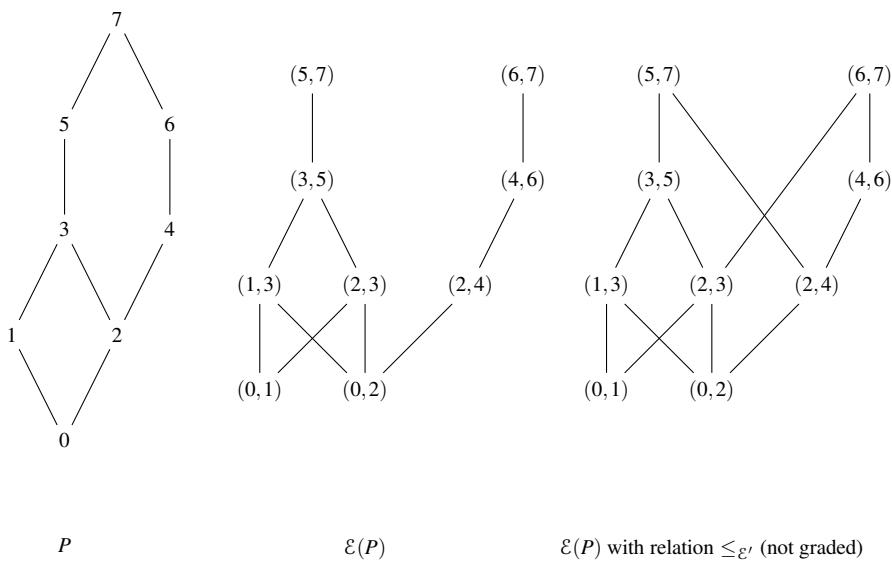
\begin{figure}[h!]
\begin{center}
\[
\raisebox{1mm}{
\begin{tikzpicture}[scale=.7] at (0,0)
  \node (0) at (0,0) {$0$};
  \node (1) at (-1,2) {$1$};
  \node (2) at (1,2) {$2$};
  \node (3) at (0,4) {$3$};
  \node (4) at (2,4) {$4$};
  \node (5) at (0,6) {$5$};
  \node (6) at (2,6) {$6$};
  \node (7) at (1,8) {$7$};
  \draw (0)--(1);
  \draw (0)--(2);
  \draw (1)--(3);
  \draw (2)--(3);
  \draw (2)--(4);
  \draw (3)--(5);
  \draw (4)--(6);
  \draw (5)--(7);
  \draw (6)--(7);
  \node (8) at (0,-2) {$P$};
\end{tikzpicture}
} \qquad
\begin{tikzpicture}[scale = .7]
  \node (0) at (0,1) {$(0,1)$};
  \node (1) at (2,1) {$(0,2)$};
  \node (2) at (0,3) {$(1,3)$};
  \node (3) at (2,3) {$(2,3)$};
  \node (4) at (4,3) {$(2,4)$};
  \node (5) at (1,5) {$(3,5)$};
  \node (6) at (5,5) {$(4,6)$};
  \node (7) at (1,7) {$(5,7)$};
  \node (8) at (5,7) {$(6,7)$};
  \draw (0)--(2);
  \draw (1)--(3);
  \draw (1)--(4);
  \draw (2)--(5);
  \draw (3)--(5);
  \draw (4)--(6);
  \draw (5)--(7);
  \draw (6)--(8);
  \draw (0)--(3);
  \draw (1)--(2);
  \node (9) at (2,-2) {$\mathcal E(P)$};
\end{tikzpicture}
\begin{tikzpicture}[scale = .7]
  \node (0) at (0,1) {$(0,1)$};
  \node (1) at (2,1) {$(0,2)$};
  \node (2) at (0,3) {$(1,3)$};
  \node (3) at (2,3) {$(2,3)$};
  \node (4) at (4,3) {$(2,4)$};
  \node (5) at (1,5) {$(3,5)$};
  \node (6) at (5,5) {$(4,6)$};
  \node (7) at (1,7) {$(5,7)$};
  \node (8) at (5,7) {$(6,7)$};
  \draw (0)--(2);
  \draw (1)--(3);
  \draw (1)--(4);
  \draw (2)--(5);
  \draw (3)--(5);
  \draw (4)--(6);
  \draw (5)--(7);
  \draw (6)--(8);
  \draw (4)--(7);
  \draw (3)--(8);
  \draw (0)--(3);
  \draw (1)--(2);
  \node (9) at (2,-2) {$\mathcal E(P)$ with relation $\le_{\mathcal{E^\prime}}$\text{ (not graded)}};
\end{tikzpicture}
\]

\caption{\label{fig:EP_definition_example}Examples of $\mathcal E$}
\end{center}
\end{figure}
\end{eg}

We observe that when $P$ has a nice structure, $\mathcal E(P)$ commonly has a nice structure as well.  In particular we examine the \textit{boolean algebra of rank $n$}, denoted $B_n$, which is defined to be the poset whose elements are subsets of $\{1,\ldots,n\}$ with the relation given by containment.  That is, for all $x,y\in B_n$, $x\le y$ if $x$ is a subset of $y$.

In this paper, we study the property of Peckness, as defined below
in Definition ~\ref{defn:peck}. The name 
``Peck'' was coined in Stanley's paper ~\cite{quotients_stanley}. However, Peck posets were studied prior to Stanley's article ~\cite{quotients_stanley}, for example in
~\cite{weyl_groups_stanley}. Peckness is a nice property
as it enjoys several equivalent definitions.  For example, see
~\cite[Lemma 1.1]{weyl_groups_stanley} and ~\cite{proctor}.

Throughout the paper we say that a group $G$ acts on $P$ if it acts on the elements of $P$ and the action is rank-preserving and order-preserving.  That is, for all $g \in G$ we have $\rk(gx) = \rk(x)$ and $x \leq y \iff gx \leq gy$. 
By a result of Stanley [Theorem \ref{thm:quotients_of_unitary_peck_posets}]{\tiny } and the fact that $B_n$ is unitary Peck, if $G$ is any action on $B_n$, then $B_n/G$ is Peck.  We conjecture the following.

\begin{conj}\label{conj:F_of_BnG_Peck}
If $G \subseteq \Aut(B_n)$, then $\mathcal E(B_n/G)$ is Peck.
\end{conj}

In the rest of the paper we give theoretical evidence for the conjecture. More precisely, we prove that the conjecture holds whenever the group action of $G$ on $B_n$ has the \emph{common cover transitive} property, which we introduce in the following definition.

\begin{defn}
\label{defn:cover_transitive}
A group action of $G$ on $P$ is \textit{common cover transitive} (CCT) if whenever $x,y,z\in P$ are such that $x\lessdot z$, $y\lessdot z$, and $y\in Gx$, then there exists some $g\in \Stab(z)$ such that $g\cdot x = y$.
\end{defn}

\begin{restatable}{thm}{cctpeck}
\label{thm:cover_transitive_implies_Peck}
If a group action of $G$ on $B_n$ is CCT, then $\mathcal E(B_n/G)$ is Peck.
\end{restatable}

The theorem trivially implies the following.
\begin{cor}
	$\mathcal E(B_n)$ is Peck.
\end{cor}

\begin{rmk}
	 In fact, it is true that $\mathcal E(B_n)$ is unitary Peck. For a proof, see Section 8 of our project report for the University of Minnesota at Twin Cities REU. \footnote{See http://www.math.umn.edu/reiner/REU/HemmingerLandesmanYao2014.pdf for the full report. This will be referred to throughout the paper as ``the REU report.''} \end{rmk}

We have found many group actions on $B_n$ that have the CCT property.  We first prove that some basic group actions on $B_n$ are CCT.  Throughout the paper we let a subgroup $G\subseteq S_n$ act on $B_n$ by letting it act on the elements within subsets of $[n]:= \{1,\ldots, n\}$, i.e. $g\cdot x = \{g\cdot i\colon i\in x\}$  for all $g\in G$, $x\in B_n$.  We also embed the dihedral group $D_{2n}$ into $S_n$ by letting it act as rotations and reflections on the vertices of an $n$-gon.

\begin{restatable}{prop}{building}
\label{prop:cover_transitive_building_blocks}
 Let $n$ be a positive integer and $p$ be a prime.  The following actions are CCT.
\begin{enumerate}
\item The action of $S_n$ on $B_n$.
\item The action of $D_{2p}$ on $B_p$.
\item The action of $D_{4p}$ on $B_{2p}$.
\end{enumerate} 
\end{restatable}

We further show that common cover transitivity is preserved under semidirect products, allowing us to describe several large families of CCT actions in Subsection \ref{ssec:CCT_examples}.

\begin{restatable}{prop}{semidirect}
\label{prop:semidirect_product_preservation}
Let $G\subseteq \Aut(P)$, $H\triangleleft G$, and $K\subset G$ such that $G = H\rtimes K$.  Suppose that the action of $H$ on $P$ is CCT and the action of $K$ on $P/H$ is CCT. Then the action of $G$ on $P$ is CCT.
\end{restatable}

The paper is organized as follows. In Section \ref{sec:background} we cover necessary background definitions for posets and Peck posets.  In Section \ref{sec:functor_of_edges} we show that $\mathcal E$ is well-defined and prove Theorem \ref{thm:cover_transitive_implies_Peck} regarding CCT actions along with various other nice properties of $\mathcal E$. Section \ref{sec:cover_transitive} contains the proofs of Propositions \ref{prop:cover_transitive_building_blocks} and \ref{prop:semidirect_product_preservation} as well as some examples of families of group actions shown to be CCT by these propositions. In Section \ref{sec:wreath_product}, we obtain a different proof of \cite[Theorem 1.1]{pak}, in the case that $r = 1$.

%%%%%%%%%%%%%%%%%  Background   %%%%%%%%%%%%%%%%%%
%%%%%%%%%%%%%%%%%%%%%%%%%%%%%%%%%%%%%%%%%%%%%
\section{Background}\label{sec:background}

In this section we review necessary background definitions for this paper.
As references for this material, see
\cite[Chapter 4]{stanley_alg_comb} and \cite{quotients_stanley}.

A {\it graded poset} $P$ is a poset with a rank function $\rk\colon P \rightarrow \mathbb Z_{\geq 0}$ satisfying the following conditions.
\begin{enumerate}
  \item If $x\in P$ and $x\lessdot y$, then $rk(x) + 1 = \rk(y)$.
  \item If $x < y$, then $\rk(x) < \rk(y)$. 
\end{enumerate}

\begin{rmk}
Note that the second condition follows from the first in the case that the poset $P$ is finite.
\end{rmk}

We denote the $i$th rank of $P$ by  $P_i = \{x \in P\colon\rk(x) = i\}$.  If for all $x\in P$ we have $0 \leq \rk(x) \leq n$, and there exist $y,z$ with $\rk(y) = 0$ and $\rk(z) = n$, we say that $P$ is a graded poset of {\it rank $n$}.

\begin{rmk}
Throughout the paper we write $x\le_P y$ to denote that $x$ is less than or equal to $y$ under the relation defined on the poset $P$.  When the poset is clear we omit the $P$ and simply write $x\le y$.  
\end{rmk}
Let $P, Q \in \mathcal P$ be two finite graded posets.  A map $f\colon P\rightarrow Q$ is a \textit{morphism} from $P$ to $Q$ if it is rank-preserving and order-preserving.  In other words, $f$ is a morphism if for all $x,y\in P$, $\rk(x) = \rk(f(x))$ and $x\le_P y $ implies $f(x)\le_Q f(y)$.  We say that $f$ is \textit{injective/surjective/bijective} if it is an injection/surjection/bijection from $P$ to $Q$ as sets.

\begin{rmk}\label{rem:bijective_morphism_not_isomorphism}
Note that we do not require the implication that $f(x)\le_Q f(y)$ implies $x\le_P y$ in order for $f$ to be a morphism.  In particular this means that a bijective morphism $f$ need not be an isomorphism, since it will not necessarily have a two-sided inverse.  
\end{rmk}

In what follows let $P$ be a poset of rank $n$, and write $p_i = |P_i|$.  If we have
$$p_0\le p_1\le \ldots \le p_k \ge p_{k+1} \ge\ldots \ge p_n$$
for some $0\le k\le n$, then $P$ is \textit{rank-unimodal}. If $p_i = p_{n-i}$ for all $1\le i\le n$, then $P$ is \textit{rank-symmetric}.  An \textit{antichain} in $P$ is a set of elements in $P$ that are pairwise incomparable.  If no antichain in $P$ is larger than the largest rank of $P$, then $P$ is \textit{Sperner}.  More generally, $P$ is \textit{$k$-Sperner} if no union of $k$ disjoint antichains in $P$ is larger than the union of the largest $k$ ranks of $P$. We say that $P$ is \textit{strongly Sperner} if it is $k$-Sperner for all $1\le k\le n$.

\begin{defn}
	\label{defn:peck}
A graded poset $P$ is \textit{Peck} if $P$ is rank-symmetric, rank-unimodal, and strongly Sperner.
\end{defn}

Let $V(P)$ and $V(P_i)$ be the complex vector spaces with bases $\{x :x\in P\}$ and $\{x :x\in P_i\}$ respectively. In determining whether $P$ is Peck, it is often useful to consider certain linear transformations on $V(P)$.

\begin{defn}
\label{defn:lefschetz}
A linear map $U\colon V(P)\rightarrow V(P)$ is an \textit{order-raising operator} if $U(V(P_n)) = 0$ and for all $0\le i\le n-1$, $x\in P_i$ we have

$$U(x) = \sum_{y\gtrdot x} c_{x,y}y$$

\noindent for some constants $c_{x,y}\in \mathbb{C}$.  We say that $U$ is the \textit{Lefschetz map} if all $c_{x,y}$ on the right hand side are equal to 1.
\end{defn}

\noindent We then have the following well-known characterization of Peck posets.

\begin{lem}[\cite{weyl_groups_stanley}, Lemma 1.1]\label{lem:Peck_poset_characterization}
A graded poset $P$ is Peck if and only if there exists an order-raising operator $U$ such that for all $0\le i < \frac{n}{2}$, the map $U^{n-2i}\colon V(P_i)\rightarrow V(P_{n-i})$ is an isomorphism.
\end{lem}

\begin{defn}
If the Lefschetz map satisfies the condition for $U$ in Lemma \ref{lem:Peck_poset_characterization}, then $P$ is \textit{unitary Peck}.
\end{defn}

Note that a group $G$ acts on $P$ if the action defines an embedding $G\hookrightarrow \Aut(P)$. We define the \textit{quotient poset} $P/G$ to be the poset whose elements are the orbits of $G$, with the relation $\mathcal{O}\le \mathcal{O}^\prime$ if there exist $x\in \mathcal{O}$, $x^\prime\in \mathcal{O}^\prime$ such that $x\le_{P} x^\prime$.  We will use the following result in the paper.

\begin{thm}[\cite{quotients_stanley}, Theorem 1]
\label{thm:quotients_of_unitary_peck_posets}
If $P$ is unitary Peck and $G\subseteq\Aut(P)$, then $P/G$ is Peck.
\end{thm}

%%%%%%%%%%%%%%%%% The Edge Functor  %%%%%%%%%%%%%%%%%%
%%%%%%%%%%%%%%%%%%%%%%%%%%%%%%%%%%%%%%%%%%%%% 

\section{The Edge Poset}
\label{sec:functor_of_edges}

In Subsection \ref{ssec:functoriality} we show that $\mathcal E$ as described in Definition \ref{defn:functor_of_edges} is well-defined and prove some useful properties of $\mathcal E$.  In Subsection \ref{ssec:dual_posets} we prove that $\mathcal{E}$ sends self-dual posets to self-dual posets.  In Subsection \ref{ssec:equivalent_defs} we give several equivalent definitions for CCT actions, and Subsection \ref{ssec:proof_of_cover_transitive_implies_Peck} is devoted to the proof of Theorem \ref{thm:cover_transitive_implies_Peck}.

\ssec{Functoriality of $\mathcal E$ and Group Actions}
\label{ssec:functoriality}
First we show that $\mathcal{E}$ is well-defined in Lemmas \ref{lem:f_partial_order}, \ref{lem:FP_graded_poset}, and \ref{lem:Ff_poset_morphism}.  We then define a natural $G$ action on $\mathcal E(P)$ and obtain a surjection $\mathcal E(P)/G\rightarrow \mathcal E(P/G)$, which are key ingredients of the proof of Theorem \ref{thm:cover_transitive_implies_Peck}.

\begin{rmk}
When the poset $P$ is clear, we will use $\leq_{\mathcal E}$ and $\lessdot_{\mathcal E}$ to refer to $\leq_{\mathcal E(P)}$ and $\lessdot_{\mathcal E(P)}$. Similarly, in Subsection ~\ref{ssec:proof_of_cover_transitive_implies_Peck}, we define posets $\mathcal H(B_n)$, and will use $\leq_{\mathcal H}$ and $\lessdot_{\mathcal H}$ in place of $\leq_{\mathcal H(B_n)}$ and $\lessdot_{\mathcal H(B_n)}$.
\end{rmk}

\begin{lem}\label{lem:f_partial_order}
The relation $\le_{\mathcal E}$ defines a partial order on $\mathcal E(P)$.
\end{lem}

\begin{proof}
We have that $(x, y)\le_{\mathcal E} (x, y)$ and that $\le_{\mathcal E}$ is transitive by definition.  Suppose $(x, y)\le_{\mathcal E} (x^\prime, y^\prime)$ and $(x^\prime, y^\prime)\le_{\mathcal E} (x, y)$.  Then $x\le_P x^\prime \le_P x$ and $y\le_P y^\prime \le_P y$, so $x = x^\prime$ and $y=y^\prime$ by antisymmetry of $\le_P$, hence $(x, y) = (x^\prime, y^\prime)$.
\end{proof}

\begin{lem}\label{lem:FP_graded_poset}
For $P$ a graded poset, the poset $\mathcal E(P)$ is graded.
\end{lem}

\begin{proof}
To show $\mathcal E(P)$ is graded, we must show that $(x, y) \lessdot_{\mathcal E} (x^\prime, y^\prime) \implies \rk(x, y)+1 = \rk(x^\prime , y^\prime)$.  This fact follows immediately from the definition of $\lessdot_{\mathcal E}$ and the definition $\rk_{\mathcal E}(x, y) = \rk_P(x)$.
\end{proof}

\begin{lem}\label{lem:Ff_poset_morphism}
Let $f\colon P\rightarrow Q$ be a morphism of finite graded posets, and define a map $\mathcal E(f)\colon \mathcal E(P)\rightarrow \mathcal E(Q)$ by $\mathcal{E}(f)(x,y) = (f(x),f(y))$ for all $(x,y)\in \mathcal{E}(P)$.  Then

\begin{enumerate}
\item $\mathcal{E}(f)$ is a morphism of finite graded posets,
\item $\mathcal{E}(\id_P) = \id_{\mathcal{E}(P)}$, and
\item if $g\colon Q\rightarrow R$ is a morphism of finite graded posets, then $\mathcal E(g\circ f) = \mathcal E(g)\circ\mathcal E(f)$.
\end{enumerate}
\end{lem}

\begin{proof}
First, we show (1).
Observe that $\mathcal E(f)$ is rank-preserving, since for all $(x, y)\in \mathcal E(P)$ we have 
$$\rk_{\mathcal E(P)}(x, y) = \rk_P(x) = \rk_Q(f(x)) = \rk_{\mathcal E(Q)}(\mathcal E(f)(x, y)).$$
Suppose $(x, y)\lessdot_{\mathcal E(P)} (x^\prime, y^\prime)$.  Then $x\lessdot_P x^\prime$ and $y\lessdot_P y^\prime$, and since $f$ is order-preserving and rank-preserving, it follows that $f(x)\lessdot_Q f(x^\prime)$ and $f(y)\lessdot_Q f(y^\prime)$. Hence $\mathcal E(f)(x,y) \lessdot_{\mathcal E(Q)} \mathcal E(f)(x^\prime , y^\prime)$. Since $\le_{\mathcal E(Q)}$ is the transitive closure of $\lessdot_{\mathcal E(Q)}$, we similarly obtain that $\mathcal E(f)$ is order-preserving and hence a morphism of finite graded posets.

Next, (2) is trivial.

Finally, we show (3). For all $(x, y)\in \mathcal E(P)$ we have 
\[
	\pushQED{\qed}
	\mathcal E(g\circ f)(x, y) = (g(f(x)), g(f(y))) = \left(\mathcal E(g)\circ\mathcal E(f)\right)(x, y).\hfill \qedhere
\]
\end{proof}

\begin{rmk}
By Lemmas \ref{lem:f_partial_order}, \ref{lem:FP_graded_poset}, and \ref{lem:Ff_poset_morphism}, the edge poset construction $\mathcal{E}$ defines an endofunctor on the category of finite graded posets with rank-preserving morphisms.
\end{rmk}

An action of $G$ on $P$ naturally induces an action of $G$ on $\mathcal E(P).$ Each element $g \in G$ is an automorphism of $P$, so $\mathcal E(g)$ is an automorphism of $\mathcal E(P)$. Lemma \ref{lem:Ff_poset_morphism} guarantees that this is a well-defined group action.

\begin{defn}\label{note:G_action_on_FP}
Given a $G$-action on $P$, define a $G$-action on $\mathcal E(P)$ by $g\cdot (x,y) = \mathcal{E}(g)(x,y) = (gx,gy)$.
\end{defn}

We then obtain a quotient poset $\mathcal E(P)/G$.  It is natural to ask whether the operation of quotienting out by $G$ commutes with $\mathcal E$, that is, whether $\mathcal E(P/G) \cong \mathcal E(P)/G$.  Unfortunately the two posets are rarely isomorphic, but there is always a surjection $\mathcal E(P)/G\rightarrow \mathcal E(P/G)$, and this surjection is also an injection precisely when the $G$-action on $P$ is CCT, as will be shown in Lemma \ref{lem:cover_transitive_equivalence}.

\begin{prop}\label{prop:surjection_between_F_quotients}
The map $q\colon \mathcal E(P)/G\rightarrow \mathcal E(P/G)$ defined by $q(G(x, y)) = (Gx,Gy)$ is a surjective morphism.
\end{prop}

\begin{proof}

Note that $q$ is well-defined because if $(x^\prime, y^\prime) = g(x, y) = (g\cdot x, g\cdot y)$ for some $g\in G$, then $x^\prime\in Gx$ and $y^\prime\in Gy$.  Clearly $q$ is rank-preserving and surjective, so it suffices to show that $q$ is order-preserving.  Suppose that $G(x, y) \lessdot_{\mathcal E(P)/G} G(w, z)$.  Then there exist some $(x_0, y_0)\in G(x, y)$, $(w_0, z_0)\in G(w, z)$ such that $x_0\lessdot_P w_0$ and $y_0\lessdot_P z_0$.  We then have that $(Gx, Gy) \lessdot_{\mathcal E(P/G)} (Gw, Gz)$ by definition. Since $\le_{\mathcal E(P/G)}$ is the transitive closure of $\lessdot_{\mathcal E(P/G)}$, $q$ is order-preserving.
\end{proof}

\ssec{The Opposite Functor and Self-Dual Posets}
\label{ssec:dual_posets}

Next, we discuss the notion of a dual poset, given by applying the opposite functor, $\op$, to a graded poset. We will show that $\op$ commutes with $\mathcal E$. This will imply that $\mathcal E(P)$ is self-dual if $P$ is, which in turn will imply that $\mathcal E(B_n/G)$ is self-dual for any group action of $G$ on $B_n$.

\begin{defn}
Let $\mathcal P$ be the category of graded posets and let $\op\colon\mathcal P \rightarrow \mathcal P$ be the opposite functor, defined on posets as follows. For $P$ a poset, the elements of $P^{\op}$ are the same as those of $P$ with order relation $\le_{P^{\op}}$ defined by $x \leq_{P^{\op}} y \iff x \geq_P y$. Induced maps on morphisms are given as follows: for $P,Q$ graded posets with $f\colon P \rightarrow Q$, then $f^{\op}\colon P^{\op} \rightarrow Q^{\op}$ is defined by $f^{\op}(x) = f(x)$. The poset $P^{\op}$ is called the {\it dual} poset of $P$. A poset $P$ is {\it self-dual} if there is an isomorphism of posets $P \cong P^{\op}$.
\end{defn}

\begin{rmk}
Note that it is easy to check $\op\colon\mathcal P \rightarrow \mathcal P$ is indeed a covariant functor. In more abstract terms, if we view $P$ as a category, then $P^{\op}$ is the opposite category. Additionally, $\op$ as defined in this way is actually an endofunctor on the category of all finite posets, which restricts to a functor on the subcategory of graded posets.
\end{rmk}

\begin{lem}
\label{lem:commuting_op_e}
The functor $\op\colon\mathcal P \rightarrow \mathcal P$ commutes with the functor $\mathcal E\colon\mathcal P \rightarrow \mathcal P$. That is, $\mathcal E(P^{\op}) \cong \mathcal E(P)^{\op}$.
\end{lem}

\begin{proof}
Define the morphism  $F\colon\mathcal E(P^{\op}) \rightarrow \mathcal E(P)^{\op}$ by sending an edge between two vertices $x$ and $y$ of $P^{\op}$ to the edge between same two vertices $x$ and $y$ of $P$. The inverse to $F$ is given by $G\colon\mathcal E(P)^{\op} \rightarrow \mathcal E(P^{\op})$ sending an edge between two vertices $x$ and $y$ of $P$ to the edge between the same two vertices $x$ and $y$ of $P^{\op}$. 
\end{proof}

\begin{prop}
\label{prop:self_dual_preservation}
If $P$ is a self-dual poset, then $\mathcal E(P)$ is also self-dual.
\end{prop}
\begin{proof}
Since $P$ is self-dual, there is an isomorphism $f\colon P \rightarrow P^{\op}$. By the  functoriality of $\mathcal E$, as shown in Lemma \ref{lem:Ff_poset_morphism}, we obtain that $\mathcal E(f)\colon\mathcal E(P) \rightarrow \mathcal E(P^{\op})$ is an isomorphism. By Lemma \ref{lem:commuting_op_e}, there is an isomorphism $\mathcal E(P^{\op}) \cong \mathcal E(P)^{\op}$. Then, let 
$
F\colon\mathcal E(P^{\op}) \rightarrow \mathcal E(P)^{\op}
$
be the same isomorphism defined in the proof of Lemma \ref{lem:commuting_op_e}, the composition $F\circ \mathcal E(f)\colon\mathcal E(P) \rightarrow \mathcal E(P)^{\op}$ defines an isomorphism, so $\mathcal E(P)$ is self-dual.
\end{proof}

\begin{eg}
While $\mathcal{E}(P)$ is often Peck for known Peck posets $P$, $\mathcal E(P)$ need not be Peck in general.  Furthermore, adding the condition that $P$ be self-dual does not change this fact.  In Figure \ref{fig:dual_not_unimodal} we give an example of a poset $P$ such that $P$ is unitary Peck and self-dual, but $\mathcal{E}(P)$ is not rank-unimodal and hence not Peck.
\end{eg}

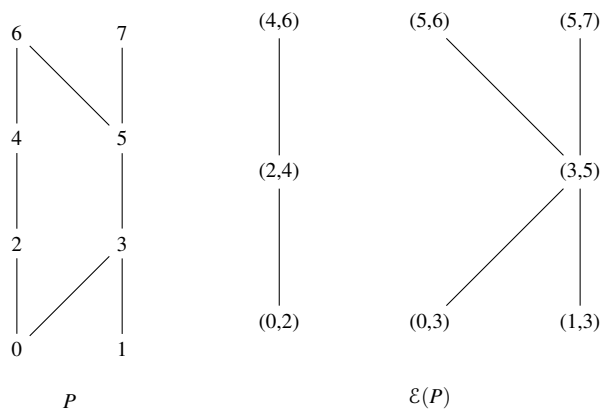
\begin{figure}[h]
\[
\begin{tikzpicture}[scale = .7]
  \node (0) at (0,0) {0};
  \node (1) at (2,0) {1};
  \node (2) at (0,2) {2};
  \node (3) at (2,2) {3};
  \node (4) at (0,4) {4};
  \node (5) at (2,4) {5};
  \node (6) at (0,6) {6};
  \node (7) at (2,6) {7};
  \draw (0)--(2);
  \draw (0)--(3);
  \draw (1)--(3);
  \draw (2)--(4);
  \draw (3)--(5);
  \draw (4)--(6);
  \draw (5)--(6);
  \draw (5)--(7);
  \node (8) at (1,-1) {$P$};
\end{tikzpicture}\qquad \hspace{10mm}
\begin{tikzpicture}
  \node (02) at (0,0) {(0,2)};
  \node (03) at (2,0) {(0,3)};
  \node (13) at (4,0) {(1,3)};
  \node (24) at (0,2) {(2,4)};
  \node (35) at (4,2) {(3,5)};
  \node (46) at (0,4) {(4,6)};
  \node (56) at (2,4) {(5,6)};
  \node (57) at (4,4) {(5,7)};
  \draw (02)--(24);
  \draw (03)--(35);
  \draw (13)--(35);
  \draw (24)--(46);
  \draw (35)--(56);
  \draw (35)--(57);
  \node (0) at (2,-1) {$\mathcal{E}(P)$};
\end{tikzpicture}\]
\caption{\label{fig:dual_not_unimodal}$P$ is self-dual and unitary Peck, but $\mathcal{E}(P)$ is not Peck.}
\end{figure}

\begin{rmk}
\label{rem:induced_action_bn}
Whenever there is an action $\psi\colon G \times [n] \rightarrow [n]$, we obtain an induced action $\phi: G \times B_n \rightarrow B_n$ defined by
$$\phi(g,\{x_1,\ldots, x_k\}) = \{\psi(g,x_1),\ldots, \psi(g,x_k)\}.$$
It is easy to see that any action $\phi\colon G \times B_n \rightarrow B_n$ arises in this way. That is, for any action $\phi$ of $G$ on $B_n$ there exists an action $\psi$ of $G$ on $[n]$ such that $\phi(g,\{x_1,\ldots, x_k\}) = \{\psi(g,x_1),\ldots, \psi(g,x_k)\}$. Whenever an action $\psi$ of $G$ on $[n]$ is given, we refer to the action $\phi$ defined above as the {\it induced action} on $B_n$.
\end{rmk}

\begin{cor}
\label{cor:duality_bn_quotients}
For any action $\phi\colon G \times B_n \rightarrow B_n$, we have that
\begin{itemize}
	\item $B_n/G$,
	\item $\mathcal E(B_n/G)$, and
	\item $\mathcal E(B_n)/G$
\end{itemize} are self-dual. In particular, they are all rank-symmetric 
\end{cor}
\begin{proof}
	By Remark \ref{rem:induced_action_bn}, any action $\phi\colon G\times B_n \rightarrow B_n$ is induced by an action $\psi\colon G \times [n] \rightarrow [n]$. 
	Define $f: B_n \rightarrow B_n^{\op}$ to be the map sending a set to its complement.
	Using this, observe that for any $\phi$, the poset $B_n/G$ is self-dual, as there is an isomorphism 
\begin{align*}
	f^G \colon B_n/G &\rightarrow (B_n/G)^{\op} \\
	G \cdot x &\mapsto G \cdot ([n] \setminus x).
\end{align*}
This map is well-defined on $G$-orbits because every action on $B_n$ is induced by an action on $[n]$. Then, by Proposition \ref{prop:self_dual_preservation}, it follows that $\mathcal E(B_n/G)$ is self-dual.

It only remains to prove that $\mathcal E(B_n)/G$ is self-dual. However, from Proposition \ref{prop:self_dual_preservation}, $\mathcal E(B_n)$ is self-dual, with the isomorphism given by 
\begin{align*}
	\mathcal E(f)\colon\mathcal E(B_n) & \rightarrow \mathcal E(B_n^{\op}) \cong \mathcal E(B_n)^{\op}\\
	(x,y) & \mapsto ([n]\setminus y,[n] \setminus x).
\end{align*}
\noindent
Once again, since the action on $B_n$ is induced by an action on $[n]$, this isomorphism descends to an isomorphism $\mathcal E(f)^G\colon\mathcal E(B_n)/G \rightarrow  (\mathcal{E}(B_n)/G)^{\op}$. Thus, $\mathcal E(B_n)/G$ is self-dual.
\end{proof}

\ssec{Equivalent Definitions of Common Cover Transitive Actions}
\label{ssec:equivalent_defs}

We next give four equivalent definitions of CCT actions. 

\begin{lem}
\label{lem:cover_transitive_equivalence}
Let $G$ be a group acting on a graded poset $P$. The following are equivalent:
\begin{enumerate}
  \item The action of $G$ on $P$ is CCT.
  \item Whenever $x \lessdot y,x \lessdot z$, and $y \in Gz$, there exists some $g \in \Stab(x)$ with $gy = z$.
  \item The map $q\colon \mathcal E(P)/G\rightarrow \mathcal E(P/G)$ defined by $q(G(x, y)) = (Gx,Gy)$ is a bijective morphism (but not necessarily an isomorphism).
  \item For all $i$ there is an equality $|(\mathcal E(P)/G)_i|=| (\mathcal E(P/G))_i|$.
\end{enumerate}
\end{lem}
\begin{proof}
First, we show $(1) \iff (3)$. By Proposition \ref{prop:surjection_between_F_quotients} we know that $q$ is a bijection exactly when there do not exist distinct orbits $G(x, y) \ne G(x^\prime, y^\prime)$ with $x^\prime\in Gx$, $y^\prime\in Gy$.  Fix $(x, y), (x^\prime, y^\prime)\in \mathcal E(P)$ such that $x^\prime\in Gx$ and $y^\prime\in Gy$.  Pick a $g\in G$ such that $g\cdot y^\prime = y$.  Then $(g\cdot x^\prime, y)\in G(x^\prime, y^\prime)$, so $G(x, y) = G(x^\prime, y^\prime)$ if and only if there exists some $g^\prime\in G$ such that $g^\prime\cdot x = g\cdot x^\prime$ and $g^\prime\cdot y = y$. Hence $q$ is a bijection if and only if the $G$ action is CCT.

Second, $(2) \iff (3)$ by an analogous argument to $(1) \iff (3)$.
Finally, we check $(3)\iff (4)$. Again using Proposition ~\ref{prop:surjection_between_F_quotients}, the morphism $q$ is always surjective. Since a morphism is always rank-preserving, it must map $(\mathcal E(P)/G)_i$ surjectively onto $(\mathcal E(P/G))_i$. However, since the posets are finite, this surjection is a bijection if and only if the sets have the same cardinality.
\end{proof}

\begin{rmk}
While $q$ is a bijection if and only if the action of $G$ on $P$ is CCT, it is {\it not} true that if the action of $G$ on $P$ is CCT, then $q$ is an isomorphism.  For example, take $G=D_{20} \subset S_{10}$ acting by reflections and rotations on $\{1,2,\ldots,10\}$, and consider the induced action on $B_{10}$. By Proposition ~\ref{prop:cover_transitive_building_blocks}, this action is CCT. However, consider $x = \{2,4\},y = \{1,2,4\},a = \{2,4,7\}$, and $b = \{2,4,6,7\}$. We may observe that $(x , y),(a, b) \in \mathcal E(B_{10})$ and $Gx < Ga, Gy < Gb$, so $(Gx, Gy) <_{\mathcal E(P/G)} (Ga, Gb)$. However, it is not true that $G(x, y)<_{\mathcal E(P)/G} G(a,b)$. \end{rmk}

\ssec{Proof of Theorem \ref{thm:cover_transitive_implies_Peck}}\label{ssec:proof_of_cover_transitive_implies_Peck}

In this section we prove Theorem \ref{thm:cover_transitive_implies_Peck}, which we recall here:

\cctpeck*

The proof is largely based on the following Lemma.

\begin{lem}\label{lem:bijection_peck_implication}
Let $P,Q$ be two graded posets with a morphism $f\colon P\rightarrow Q$ that is a bijection (but not necessarily an isomorphism). If $P$ is Peck, then $Q$ is Peck.
\end{lem}
\begin{proof}
Let $\rk(P) = \rk(Q) = n$.  Since $P$ is Peck there exists an order-raising operator $U$ such that $U^{n-2i}\colon V(P_i)\rightarrow V(P_{n-i})$ is an isomorphism.  Since $f$ is a poset morphism, it follows that the map $f\circ U\circ f^{-1}$ is an order-raising operator on $Q$.  We then have that $f\circ U^{n-2i}\circ f^{-1} = \left(f\circ U\circ f^{-1}\right)^{n-2i}\colon V(Q_i)\rightarrow V(Q_{n-i})$ is an isomorphism since $U^{n-2i}\colon V(P_i)\rightarrow V(P_{n-i})$ is an isomorphism and $f$ is a bijection.

\end{proof}

By Lemmas ~\ref{lem:cover_transitive_equivalence} and ~\ref{lem:bijection_peck_implication}, in order to prove Theorem \ref{thm:cover_transitive_implies_Peck} it suffices to prove that $\mathcal E(B_n)/G$ is Peck.  One way to do this is to prove that $\mathcal E(B_n)$ is unitary Peck and then apply Theorem \ref{thm:quotients_of_unitary_peck_posets}.  In fact, this approach generalizes to an arbitrary poset $P$.

\begin{thm}
\label{edge_unitary_peck_quotient}
If the action of $G$ on $P$ is CCT and $\mathcal E(P)$ is unitary Peck, then $\mathcal E(P/G)$ is Peck.
\end{thm}
\begin{proof}
Since the $G$-action is CCT, there is a bijection $q\colon\mathcal{E}(P)/G \rightarrow \mathcal{E}(P/G)$ by Lemma ~\ref{lem:cover_transitive_equivalence}.  Since $\mathcal{E}(P)$ is unitary Peck we have that $\mathcal{E}(P)/G$ is Peck by Theorem ~\ref{thm:quotients_of_unitary_peck_posets}, hence $\mathcal{E}(P/G)$ is Peck by Lemma ~\ref{lem:bijection_peck_implication}.
\end{proof}

We prove that $\mathcal E(B_n)$ is unitary Peck for $n > 2$ in Section 8 of the REU report, but unfortunately the proof is technical and computational. Note that by Theorem \ref{edge_unitary_peck_quotient}, this immediately implies Theorem \ref{thm:cover_transitive_implies_Peck}. Fortunately there is a cleaner -- albeit less direct -- route to proving Theorem \ref{thm:cover_transitive_implies_Peck}. In order to avoid showing that $\mathcal E(B_n)$ is unitary Peck, we define a graded Peck poset $\mathcal{H}(B_n)$ which injects into $\mathcal E(B_n)$. 
\begin{defn}
\label{defn:h_map}
For $P$ a graded poset, define the graded poset $\mathcal H(P)$ as follows.  Let the elements $(x, y) \in \mathcal H(P)$ be pairs $(x,y) \in P\times P$ such that $x \lessdot y$.  Define $(x, y) \lessdot_{\mathcal H} (x^\prime, y^\prime)$ if $x \lessdot x^\prime,y\lessdot y^\prime$ and $x^\prime \neq y$.  Then define $\leq_{\mathcal H}$ to be the transitive closure of $\lessdot_{\mathcal H}$, and define $\rk_{\mathcal H}(x, y) = \rk_P(x)$.
\end{defn}

\begin{eg}
\label{eg:3boolean}
We give an example of the poset $\mathcal H(B_3)$ in Figure \ref{fig:3boolean}. Observe that $\mathcal H(B_3)$ can be written as a disjoint union of three copies of $B_2$. This is a single case of the more general phenomenon proven in Proposition \ref{prop:computing_HBn}.
\end{eg}

\begin{figure}[h!]
\[\begin{tikzpicture}[scale = 1.5]
  \node (0) at (0,0) {$\emptyset$};
  \node (1) at (-1,1) {$\{1\}$};
  \node (2) at (0,1) {$\{2\}$};
  \node (3) at (1,1) {$\{3\}$};
  \node (4) at (-1,2) {$\{1,2\}$};
  \node (5) at (0,2) {$\{1,3\}$};
  \node (6) at (1,2) {$\{2,3\}$};
  \node (7) at (0,3) {$\{1,2,3\}$};
  \draw (0)--(1);
  \draw (0)--(2);
  \draw (0)--(3);
  \draw (1)--(4);
  \draw (1)--(5);
  \draw (2)--(4);
  \draw (2)--(6);
  \draw (3)--(5);
  \draw (3)--(6);
  \draw (4)--(7);
  \draw (5)--(7);
  \draw (6)--(7);
  \node (8) at (0,-.5) {$B_3$};
\end{tikzpicture} \]

\vspace*{0.3in}

\[\begin{tikzpicture}[scale = 1]
  \node (0) at (0,0) {$(\emptyset,\{1\})$};
  \node (1) at (4,0) {$(\emptyset,\{2\})$};
  \node (2) at (8,0) {$(\emptyset,\{3\})$};
  \node (3) at (3,1) {$(\{1\},\{1,2\})$};
  \node (4) at (7,1) {$(\{1\},\{1,3\})$};
  \node (5) at (-1,1) {$(\{2\},\{1,2\})$};
  \node (6) at (9,1) {$(\{2\},\{2,3\})$};
  \node (7) at (1,1) {$(\{3\},\{1,3\})$};
  \node (8) at (5,1) {$(\{3\},\{2,3\})$};
  \node (9) at (8,2) {$(\{1,2\},\{1,2,3\})$};
  \node (10) at (4,2) {$(\{1,3\},\{1,2,3\})$};
  \node (11) at (0,2) {$(\{2,3\},\{1,2,3\})$};
  \draw (0)--(5)--(11)--(7)--(0);
  \draw (1)--(3)--(10)--(8)--(1);
  \draw (2)--(4)--(9)--(6)--(2);
  \node (12) at (4,-1) {$\mathcal{H}(B_3)$};
\end{tikzpicture}\]
\caption{\label{fig:3boolean} $B_3$ and $\mathcal H (B_3)$}
\end{figure}
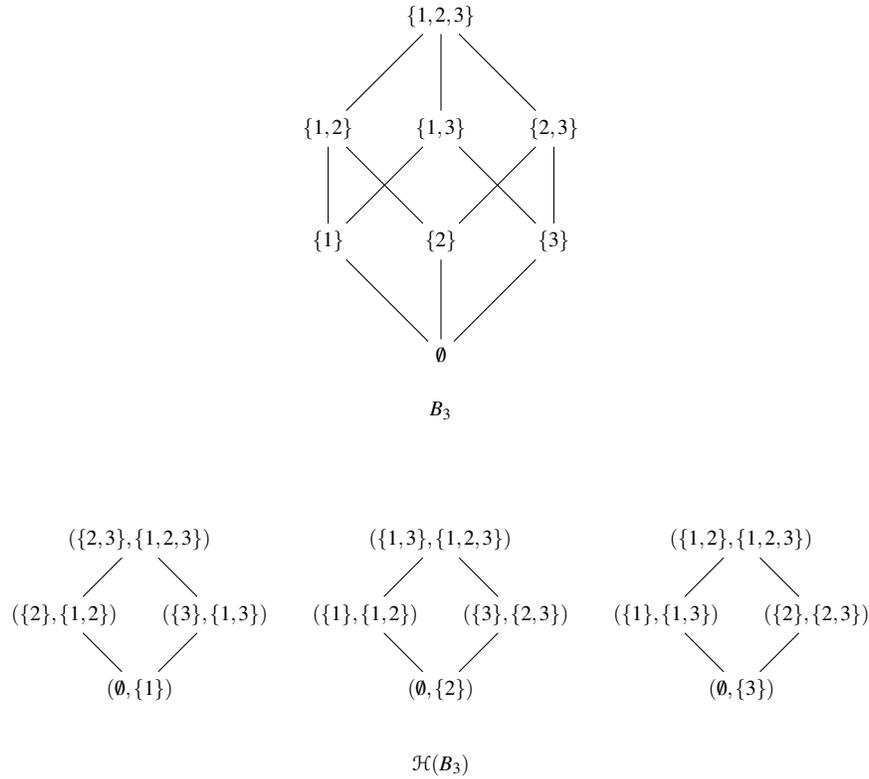

\begin{rmk}\label{rem:order_containment}
Note that by definition, $(x,y)\lessdot_{\mathcal{H}} (x^\prime,y^\prime)$ precisely when $(x,y)\lessdot_{\mathcal{E}} (x^\prime,y^\prime)$ and $x^\prime\neq y$, hence $(x, y)\lessdot_{\mathcal{H}} (x^\prime, y^\prime) \Rightarrow (x, y)\lessdot_{\mathcal E} (x^\prime, y^\prime)$.  In other words, $\mathcal{H}(P)$ has the same elements as $\mathcal{E}(P)$ but with a weaker partial order.
\end{rmk}

\begin{lem}\label{lem:HP_order}
For $P$ a graded poset, the object $\mathcal{H}(P)$, as defined in Definition ~\ref{defn:h_map}, is a graded poset.
\end{lem}

\begin{proof}
This follows immediately from Remark \ref{rem:order_containment} and the fact that $\mathcal E(P)$ is graded.
\end{proof}

\begin{rmk}
While $\mathcal E\colon\mathcal P \rightarrow \mathcal P$ is a functor, $\mathcal H$ is not a functor. In particular, it is not possible to define $\mathcal H(f)$ for $f$ a morphism. This is illustrated in Figure \ref{fig:h_morphism}. For example, suppose we took $f\colon P \rightarrow Q$ defined by $f(1) = a, f(2) = f(3) = b$, and $f(4) = c$. 
Then there is no possible morphism $\mathcal H(f)\colon\mathcal H(P)\rightarrow \mathcal H(Q)$ because there are no morphisms $\mathcal H(P) \rightarrow \mathcal H(Q)$ whatsoever. 

\begin{figure}[h!]
\begin{center}
\begin{tikzpicture}[scale=.7]
  \node (b) at (0:2cm) {$3$};
  \node (a) at (90:2cm) {$4$};
  \node (c) at (180:2cm) {$2$};
  \node (d) at (270:2cm) {$1$};
  \draw (a) -- (b) -- (d) -- (c)--(a);
  \node (e) at (270:3cm) {$P$};
\end{tikzpicture}\quad
\begin{tikzpicture}[scale=.7]
  \node (b) at (90:2cm) {$a$};
  \node (a) at (270:2cm) {$b$};
  \node (c) at (0:0cm) {$c$};
  \draw (a) -- (c) -- (b);
  \node (e) at (270:3cm) {$Q$};
\end{tikzpicture}\quad
\begin{tikzpicture}[scale=.7]
  \node (b) at (45:1.414cm) {$(3,4)$};
  \node (a) at (135:1.414cm) {$(2,4)$};
  \node (c) at (225:1.414cm) {$(1,2)$};
  \node (d) at (315:1.414cm) {$(1,3)$};
  \draw (a) -- (d) ;
  \draw (c) -- (b);
  \node (e) at (270:3cm) {$\mathcal{H}(P)$};
\end{tikzpicture} \quad
\begin{tikzpicture}[scale=.7]
  \node (b) at (90:1cm) {$(b,c)$};
  \node (a) at (270:1cm) {$(a,b)$};
  \node (e) at (270:3cm) {$\mathcal{H}(Q)$};
\end{tikzpicture}
\end{center}
\caption{\label{fig:h_morphism} A map of posets $f:P \rightarrow Q$ with no possible map $\mathcal H(f)$}
\end{figure}
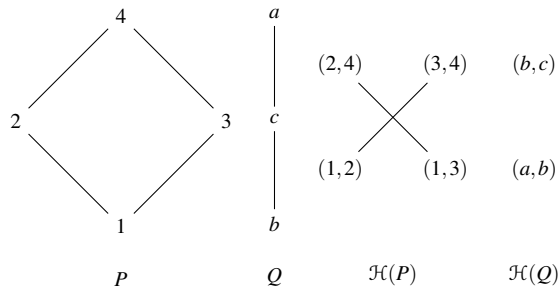
\end{rmk}

Given an action of a group $G$ on $P$, we define an action of $G$ on $\mathcal{H}(P)$ as we did for $\mathcal{E}(P)$ by again defining $g\cdot (x,y) = (gx,gy)$ for all $(x,y)\in P$.  We will then have a well-defined quotient poset $\mathcal{H}(P)/G$ with the same elements as $\mathcal{E}(P)/G$.

\begin{lem}
\label{lem:G_action_on_HP}
The automorphism defined by $g\cdot (x, y)= (gx, gy)$ for all $g\in G$, $(x, y)\in \mathcal{H}(P)$ yields a group action of $G$ on $\mathcal{H}(P)$.
\end{lem}

\begin{proof}
Let $g\in G$.  Since $\le_{\mathcal{H}}$ is the transitive closure of $\lessdot_{\mathcal{H}}$ it suffices to show that for all $(x,y),(x^\prime,y^\prime)\in \mathcal{H}(P)$ we have $(x, y) \lessdot_{\mathcal H} (x^\prime,y^\prime) \iff g(x, y) \lessdot_{\mathcal H} g(x^\prime, y^\prime)$.  Since $g$ is an automorphism of $P$, we have $x\le_P x^\prime \iff gx\le_P gx^\prime$, $y\le_P y^\prime \iff gy\le_P gy^\prime$, and $y\neq x^\prime \iff gy\neq gx^\prime$, so the result follows from the definition of $\le_{\mathcal{H}}$.
\end{proof}

\begin{lem}
\label{lem:bijection_h_f}
The map 
\begin{align*}
	f\colon\mathcal H(P)/G &\rightarrow \mathcal E(P)/G \\
	G(x,y) &\mapsto G(x,y)
\end{align*}
 is a bijective morphism for any group action of $G$ on $P$.
\end{lem}

\begin{proof}
The elements of $\mathcal H(P)/G$ and $\mathcal E(P)/G$ are the same by definition, so it suffices to show that $f$ is a morphism. Since $f$ is clearly rank-preserving, it suffices to show $f$ is order-preserving. This is immediate from Remark ~\ref{rem:order_containment}.
\end{proof}

The remaining step in the proof of Theorem \ref{thm:cover_transitive_implies_Peck} is to show that $\mathcal{H}(B_n)$ is unitary Peck, which we do by generalizing Example \ref{eg:3boolean} and showing that $\mathcal{H}(B_n)$ is isomorphic to a disjoint union of boolean algebras.

\begin{prop}\label{prop:computing_HBn}
The graded poset $\mathcal{H}(B_n)$ is isomorphic to $n$ disjoint copies of $B_{n-1}$.
\end{prop}

\begin{proof}
Let the $n$ disjoint copies of $B_{n-1}$ be labeled $B_{n-1}^{(i)}$, $1\le i\le n$, with the elements of $B_{n-1}^{(i)}$ labeled $x^{(i)}$, $x\subseteq \{1,\ldots,n-1\}$.  We will show that the map 

\begin{align*}
f\colon \mathcal{H}(B_n) &\longrightarrow \bigcup_{i=1}^n B_{n-1}^{(i)}\\
(x,x\cup{i})&\longmapsto x^{(i)}
\end{align*}

\noindent is an isomorphism.  Suppose we have $(x, y),(x', y') \in \mathcal H(B_n)$ with $(x, y) \lessdot_{\mathcal H} (x', y')$. Let $j\in [n]$ such that $y^\prime = y\cup\{j\}$, and let $i\in [n]$ such that $x^\prime = x\cup \{i\}$. If $i \ne j$, then $x' = y$, contradicting the assumption that $(x, y) \lessdot_{\mathcal H} (x', y')$. Thus $x^\prime = x\cup\{i\}$ and $y^\prime = y\cup\{i\}$ for some $i\in [n]$.

Conversely, we can easily check that if $i\not\in y$, then $(x, y)\lessdot_{\mathcal{H}} (x\cup\{i\}, y\cup\{i\})$.  It follows that for all subsets $w \subset [n]$ such that $|w| = 1$, there is an isomorphism 
\begin{align*}
\{(x, y)\colon y\setminus x = w\} & \rightarrow B_{n-1}\\
(x,y) &\mapsto (x\setminus w,y\setminus w). 
\end{align*}
Furthermore, if $y\setminus x \ne y^\prime \setminus x^\prime$, then $(x, y)$ and $(x^\prime, y^\prime)$ are incomparable, so these subposets indexed by $w$ are pairwise disjoint, and $\mathcal H(B_n)$ is isomorphic to $n$ copies of $B_{n}$.
\end{proof}

\begin{cor}\label{cor:HBn_unitary_peck}
The graded poset 
$\mathcal H(B_n)$ is unitary Peck for all $n\ge 0$.
\end{cor}

\begin{proof}
This follows immediately from Proposition \ref{prop:computing_HBn} and the fact that $B_{n-1}$ is unitary Peck.  Indeed, $B_n$ is shown to be unitary Peck in \cite[Theorem 2a]{quotients_stanley} by noting that $B_k = (B_1)^k$ and that $B_1$ is clearly unitary Peck.
Note that $\mathcal H(B_0)$ is the empty poset, so it is vacuously unitary Peck.
\end{proof}

\begin{cor}\label{cor:quotients_of_HBn_peck}
The graded poset $\mathcal H(B_n)/G$ is Peck for any subgroup $G\subset \Aut(B_n)$.
\end{cor}

\begin{proof}
This follows from Corollary \ref{cor:HBn_unitary_peck} and Theorem \ref{thm:quotients_of_unitary_peck_posets}.
\end{proof}

The next corollary will not be particularly relevant in proving Theorem \ref{thm:cover_transitive_implies_Peck}, but we note it as an aside.

\begin{cor}
Both $\mathcal{E}(B_n)$ and $\mathcal H(B_n)$ have symmetric chain decompositions (SCD).
\end{cor}

\begin{proof}
The graded poset
$\mathcal{H}(B_n)$ has an SCD by Proposition ~\ref{prop:computing_HBn} and the fact that $B_{n-1}$ has an SCD, as shown in \cite{greene}.  By Lemma ~\ref{lem:bijection_h_f} there is a bijective morphism $f\colon\mathcal{H}(B_n)\rightarrow\mathcal{E}(B_n)$, and since a bijective morphism takes an SCD to an SCD it follows that $\mathcal{E}(B_n)$ has an SCD.
\end{proof}

\begin{cor}
\label{cor:quotiented_edge_peck}
The graded poset $\mathcal E(B_n)/G$ is Peck for any subgroup $G\subset \Aut(B_n)$.
\end{cor}
\begin{proof}
By Corollary ~\ref{cor:quotients_of_HBn_peck}, $\mathcal H(B_n)/G$ is Peck. By Lemma ~\ref{lem:bijection_h_f}, the map 
\begin{align*}
f\colon\mathcal H(B_n)/G &\rightarrow \mathcal E(B_n)/G \\
G(x, y) &\mapsto G(x, y)
\end{align*}
 is a bijective morphism. Then, by Lemma ~\ref{lem:bijection_peck_implication}, it follows that $\mathcal E(B_n)/G$ is Peck.
\end{proof}

We now deduce Theorem \ref{thm:cover_transitive_implies_Peck}.

\begin{proof}[Proof of Theorem \ref{thm:cover_transitive_implies_Peck}]
By Corollary \ref{cor:quotiented_edge_peck}, $\mathcal{E}(B_n)/G$ is Peck for any group action of $G$ on $B_n$. Since the $G$-action is CCT, there is a bijective morphism from $\mathcal{E}(B_n)/G$ to $\mathcal{E}(B_n/G)$ by Lemma \ref{lem:cover_transitive_equivalence}. Hence $\mathcal{E}(B_n/G)$ is Peck by Lemma \ref{lem:bijection_peck_implication}.
\end{proof}

Note that we have also developed several generalizations of $\mathcal E$, for which many similar results hold. For more information, see Subsection 3.3 of the REU report.

%%%%%%%%%%%%%%%%%  CCT actions  %%%%%%%%%%%%%%%%%%
%%%%%%%%%%%%%%%%%%%%%%%%%%%%%%%%%%%%%%%%%%%%%

\section{Common Cover Transitive Actions}
\label{sec:cover_transitive}
In this section, we develop the theory of CCT actions $\phi$ where $G$ is a group, $P$ is a poset, and $\phi:G\times P \rightarrow P$ is an action. Recall Definition ~\ref{defn:cover_transitive}, that $\phi$ is CCT if whenever $x,y,z \in P$ such that $x\lessdot z,y\lessdot z$, and $x \in Gy$, then there exists $g \in \Stab(z)$ with $gx = y$.  We show that the CCT property is closed under semidirect products, in the appropriate sense. From Proposition ~\ref{prop:cover_transitive_building_blocks}, which will be proven in Subsection \ref{sssec:dihedral}, the action of $S_n$ on $B_n$ and the action of certain dihedral groups are CCT. We can then use these as building blocks to construct other CCT groups. In particular, we shall show in this section that automorphism groups of rooted trees are CCT.

\begin{eg}
\label{eg:trivial_edgequot}
Two rather trivial examples of CCT actions are $\phi\colon S_n\times B_n \rightarrow B_n$ and $\psi\colon G\times B_n\rightarrow B_n$ where $G$ is arbitrary, $\phi$ is the action induced by $S_n$ permuting the elements of $[n]$, and $\psi$ is the trivial action. In the former case, $\mathcal E(B_n/S_n)$ is simply a chain with $n$ points, and so is $\mathcal E(B_n)/S_n$, since all $(x, y)$ are identified under the $S_n$ action. In the latter case, since $G$ acts trivially by $\phi$ we have that $\mathcal E(B_n/G) \cong \mathcal E(B_n)$ and $\mathcal E(B_n)/G \cong \mathcal E(B_n)$. So again, $\psi$ is CCT.
\end{eg}

\ssec{Preservation Under Semidirect Products}
\label{ssec:semidirect_product_preservation}

\begin{lem}
Let $G\subseteq\Aut(P)$, $H\triangleleft G$, and $K\subset G$ such that $G = H\rtimes K$. We then have a well-defined group action 
\begin{align*}
K \times P/H &\rightarrow P/H \\
(k,Hx) &\mapsto H(k \cdot x).
\end{align*}

\end{lem}
\begin{proof}

Note that if $x,x^\prime\in Hx,$ we have $x^\prime = h\cdot x$ for some $h\in H$. Since $H$ is normal in $G,$ we have that for all $k \in G$ there exists $h^\prime \in H$ so that $khk^{-1} = h^\prime$. So 

$$k\cdot x^\prime = kh\cdot x = k(k^{-1}h^\prime k)\cdot x = h^\prime\cdot (k\cdot x)$$

Hence $k\cdot x$ and $k\cdot x^\prime$ are in the same $H$-orbit, so we have a well-defined group action of $K$ on $P/H$ defined by $k\cdot Hx = H(k\cdot x)$.
\end{proof}

Recall Proposition ~\ref{prop:semidirect_product_preservation}, which says that the CCT property is preserved under semidirect products.  We will use Proposition ~\ref{prop:semidirect_product_preservation} to construct more examples of CCT group actions, in particular using it to give a simple proof that CCT actions are preserved under direct products and wreath products.

\semidirect*
\begin{proof}
Since $G = H\rtimes K$, every element $g\in G$ can be written uniquely as a product $hk$ for some $h\in H$, $k\in K$.  Let $x,y,z\in P$ be such that $x\lessdot z$, $y\lessdot z$, and such that there exists some $h_0k_0\in G$ with $h_0k_0\cdot x = y$.  It suffices to show that there exists some $g\in \Stab_G(z)$ such that $g\cdot x = y$.

The orbits $Hx, Hy, Hz\in P/H$ satisfy $Hx\lessdot Hz$, $Hy\lessdot Hz$ such that $k_0\cdot Hx = Hy$. Thus, since the action of $K$ on $P/H$ is CCT, there exists some $k_1\in K$ such that $k_1\in \Stab_K(Hz)$ and $k_1\cdot Hx = Hy$.  It follows that there exists some $h_1\in H$ such that $h_1k_1h_0\in \Stab_G(z)$ and $h_1k_1h_0\cdot x\in Hy$.

Write $x^\prime = h_1k_1h_0\cdot x$.  Since the group action of $G$ must be order-preserving by definition, we have that $x^\prime \lessdot z$.  We already had that $y\lessdot z$ and $x^\prime\in Hy$, hence there exists some $h_2\in \Stab_H(z)$ such that $h_2\cdot x^\prime = y$ by the fact that the action of $H$ on $P$ is CCT.  Then we have that $h_2h_1k_1h_0\cdot x = h_2\cdot x^\prime = y$ and $h_2h_1k_1h_0\cdot z = h_2\cdot z = z$, as desired.
\end{proof}

\begin{prop}
\label{prop:direct_product_preservation}
If $\phi\colon G\times P\rightarrow P$ and $\psi\colon H \times Q \rightarrow Q$ are two CCT actions, then the direct product action 
\begin{align*}	
\phi \times \psi\colon(G\times H)\times (P\times Q) &\rightarrow P\times Q \\
(g,h)\cdot (x,y) &\mapsto (gx,hy)
\end{align*}
is also CCT.
\end{prop}

\begin{proof}
First note that if either $G$ or $H$ acts trivially, then it can be easily checked that the action of $G\times H$ is CCT.  Next, observe that $G \times H$ can be viewed as the semidirect product $(G\times \{e\}) \rtimes (\{e\} \times H)$. Since the action of $G$ on $P$ is CCT, the action of $G\times \{e\}$ on $P \times Q$ is CCT. Also, since the action of $H$ on $Q$ is CCT, it follows that the action of $\{e\}\times H$ on $P \times Q/(G \times \left\{ e \right\})$ is CCT. Therefore, the action of $(G\times \{e\}) \rtimes (\{e\} \times H)$ satisfies the conditions of Proposition ~\ref{prop:semidirect_product_preservation}, hence the action of $G\times H$ on $P \times Q$ is CCT.
\end{proof}

Next, we use Proposition ~\ref{prop:semidirect_product_preservation} to prove in Proposition ~\ref{prop:wreath_preservation} that the CCT property is preserved under wreath products with the symmetric group. First, we need the definition of the wreath product.

\begin{defn}
For $G$ and $H$ groups with $H \subset S_\ell$, the {\it wreath product}, denoted by $G \wr H$, is the group whose elements are pairs $(g,h) \in G^\ell\times H$ with multiplication defined by
\begin{align*}
((g_1',\ldots, g_\ell'),h') \cdot ((g_1,\ldots, g_\ell) ,h) =((g'_{h'(1)}g_1,\ldots, g'_{h'(\ell)}g_\ell),hh')
\end{align*}
where $H$ acts on $[\ell]$ via the embedding of $H$ into $S_\ell$.
\end{defn}

In other words, $G\wr H$ can be viewed as a certain semidirect product $G^\ell \rtimes H$.

\begin{defn}
\label{defn:wreath_action}
For any group $G$ with a given action $\psi\colon G\times P \rightarrow P$, we obtain an induced action $\phi\colon G \wr H \times P^\ell \rightarrow P^\ell$ defined by 
$$((g_1,\ldots, g_\ell),h)(a_1,\ldots, a_\ell) = (g_{h^{-1}(1)}\cdot a_{h^{-1}(1)},\ldots,g_{h^{-1}(\ell)} \cdot a_{h^{-1}(\ell)}).$$
\end{defn}

\begin{rmk}
Heuristically, one may think of the above action as obtained by first having $G$ act separately on the $\ell$ distinct copies of $P$, and then letting $H$ act by permuting the copies.
\end{rmk}

\begin{lem}
\label{lem:symmetric_group_product_action}
For $P$ a graded poset, the action 
\begin{align*}
\phi\colon S_\ell \times P^\ell &\rightarrow P^\ell\\
(\sigma, (x_1,\ldots, x_\ell)) &\mapsto (x_{\sigma(1)},\ldots, x_{\sigma(\ell)})
\end{align*}
is CCT.
\end{lem}

\begin{proof}
For $a \in P^\ell$ denote $a = (a_1,\ldots, a_\ell)$. Suppose $x,y,z \in P^\ell$ with $x\lessdot z,y\lessdot z$, and $x \in S_\ell y$, but $x \neq y$. This means there is a unique $i$ such that $x_i \lessdot z_i$ and $x_k = z_k$ for $k \neq i$. Additionally, there is a unique $j$ for which $y_j \lessdot z_j$ and $y_k =z_k$ for $k \neq j$. Since $x \in S_\ell y$, we obtain the equality of multisets $\{x_1,\ldots, x_\ell\}=\{y_1,\ldots,y_\ell\}$. But for $k \neq i,j$ we have $x_k = z_k = y_k$, so we also obtain equality of sets $\{x_i,x_j\} = \{y_i,y_j\}$. Since $y_j \lessdot x_j$, we obtain $y_j = x_i,y_i = x_j$. Then, taking the transposition $\sigma  = (ij) \in S_\ell$, it follows that $\sigma \in \Stab(z)$ and $\sigma \cdot x = y$.
\end{proof}

\begin{prop}
\label{prop:wreath_preservation}
If $\psi\colon G\times P \rightarrow P$ is CCT, let $\phi\colon G\wr S_\ell \times P^\ell \rightarrow P^\ell$ be the induced action defined in Definition \ref{defn:wreath_action}. Then $\phi$ is also CCT.
\end{prop}
\begin{proof}
Note that the wreath product $G \wr S_\ell$ can be viewed as a semidirect product $G^\ell \rtimes S_\ell$. Since the action of $G$ on $P$ is CCT, we obtain that the action of $G^\ell$ on $P^\ell$ is CCT by Proposition ~\ref{prop:direct_product_preservation}.  Furthermore, the action 
\begin{align*}
S_\ell \times (P/G)^\ell &\rightarrow (P/G)^\ell \\
	(\sigma ,(x_1,\ldots, x_\ell)) &\mapsto (x_{\sigma(1)}\ldots, x_{\sigma(\ell)})
\end{align*}
 for $\sigma \in S_\ell$ and $x_i \in P/G$ is CCT by Lemma ~\ref{lem:symmetric_group_product_action}.  Since $P^\ell/G^\ell \cong (P/G)^\ell$, it follows that the action $\phi$ satisfies the conditions of Proposition ~\ref{prop:semidirect_product_preservation}, so $\phi$ is CCT.
\end{proof}

%%%%%%%%%%%%% subsection: CCT examples %%%%%%%%%%% 

\ssec{Examples of CCT Actions}\label{ssec:CCT_examples}

In this subsection, we describe several classes of CCT actions. First, we show that the automorphism group of any rooted tree is CCT. Second, we show that linear automorphisms of simplices and octahedra are CCT. Third, we show that the left multiplication action is CCT if and only if the group is $\mathbb Z_2^k$, and that any action of $\mathbb Z_2^k$ on $[n]$ induces a CCT action on $B_n$. In the end of this subsection, we prove Proposition \ref{prop:cover_transitive_building_blocks}, which shows that certain symmetric group and dihedral group actions are CCT.

\sssec{An application to rooted trees}
\label{ssec:rooted_trees}
In this subsection, we prove that the automorphism group of a rooted tree is always CCT. To do this we will apply Proposition ~\ref{prop:wreath_preservation} and Proposition ~\ref{prop:direct_product_preservation}, using that the automorphism group of a  rooted tree is essentially built from direct products and wreath products with a symmetric group. To this aim, we first give definitions relating to rooted trees, then characterize their automorphisms, and finally show that such automorphism groups are always CCT.

\begin{defn}
A graded poset $P$ is a {\it rooted tree} if there is a unique element $z \in P$ of maximal rank, called the {\it root}, and for all $x \in P$ other than the root there exists a unique $y \in P$ with $y \gtrdot x$.
\end{defn}

\begin{eg}
We give two examples of rooted trees in Figures \ref{fig:tree1} and \ref{fig:tree2}.
\end{eg}
\begin{figure}[h!]
\[\begin{tikzpicture}[scale=.7]
  \node (0) at (-4,0) {$7$};
  \node (1) at (-3,0) {$8$};
  \node (2) at (-2,0) {$9$};
  \node (3) at (-1,0) {$10$};
    \node (4) at (0,0) {$11$};
  \node (5) at (1,0) {$12$};
  \node (6) at (2,0) {$13$};
  \node (7) at (3,0) {$14$};
  \node (8) at (-3.5,1.2) {$3$};
  \node (9) at (-1.5,1.2) {$4$};
  \node (10) at (.5,1.2) {$5$};
  \node (11) at (2.5,1.2) {$6$};
    \node (12) at (-2.5,2.3) {$1$};
  \node (13) at (1.5,2.3) {$2$};
  \node (14) at (-.5,3.5) {$0$};
  \draw (0)--(8);
  \draw (1)--(8);
  \draw (2)--(9);
  \draw (3)--(9);
  \draw (4)--(10);
  \draw (5)--(10);
  \draw (6)--(11);
  \draw (7)--(11);
  \draw (8)--(12);
  \draw (9)--(12);
  \draw (10)--(13);
  \draw (11)--(13);
  \draw (12)--(14);
  \draw (13)--(14);
\end{tikzpicture}\]
\caption{\label{fig:tree1} An example of a rooted tree with $8$ leaves}
\end{figure}
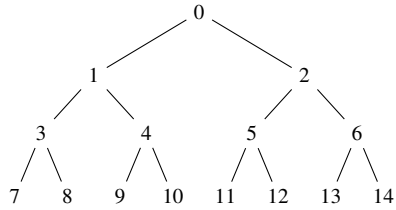

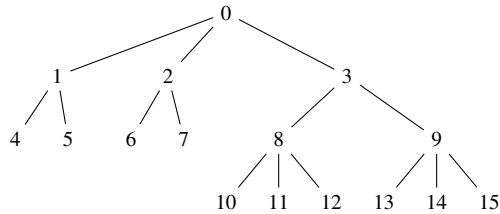
\begin{figure}[h!]
\[\begin{tikzpicture}[scale=.7]
  \node (0) at (-4,0) {$4$};
  \node (1) at (-3,0) {$5$};
  \node (2) at (-1.8,0) {$6$};
  \node (3) at (-0.8,0) {$7$};
    \node (4) at (0,-1.2) {$10$};
  \node (5) at (1,-1.2) {$11$};
  \node (6) at (2,-1.2) {$12$};
  \node (7) at (3,-1.2) {$13$};
  \node (8) at (4,-1.2) {$14$};
  \node (9) at (5,-1.2) {$15$};
  \node (10) at (1, 0) {$8$};
  \node (11) at (4,0) {$9$};
    \node (12) at (-3.2,1.2) {$1$};
  \node (13) at (-1.1,1.2) {$2$};
  \node (14) at (2.3,1.2) {$3$};
  \node (15) at (0,2.4) {$0$};
 \draw (0)--(12);
  \draw (1)--(12);
  \draw (2)--(13);
  \draw (3)--(13);
  \draw (4)--(10);
  \draw (5)--(10);
  \draw (6)--(10);
  \draw (7)--(11);
  \draw (8)--(11);
  \draw (9)--(11);
  \draw (10)--(14);
  \draw (11)--(14);
  \draw (12)--(15);
  \draw (13)--(15);
  \draw (14)--(15);
\end{tikzpicture}\]
\caption{\label{fig:tree2} An example of a rooted tree with 10  leaves}
\end{figure}

\begin{defn}
For $P$ a rooted tree, an element $x \in P$ is a {\it leaf} if there is no $z \in P$ with $x > z$. Denote the set of all leaves of $P$ by $L(P)$.
\end{defn}

We next recall a useful, elementary lemma whose proof we omit.

\begin{lem}
\label{lem:induced_tree_action}
Let $P$ be a rooted tree. Then the action of $\Aut(P)$ on $P$ induces an action of $\Aut(P)$ on $L(P)$. Furthermore, there is also an induced action of $\Aut(P)$ on $B_n$, where $n = |L(P)|$. 
\end{lem}

For the rest of this section only, fix a rooted tree $P$ and denote by $G$ the group of automorphisms $\Aut(P)$. Let $G$ act on $B_n$, where $n = |L(P)|$, by the induced action $\phi\colon G \times L(P) \rightarrow L(P)$ described in the proof of Lemma ~\ref{lem:induced_tree_action}.
For $x \in P$, denote $D(x) = \{y \in P\colon y \leq x\},$ so $D(x)$ is the maximal subposet of $P$ with maximum element $x$.

\begin{prop}
\label{prop:automorphism_trees}
Let $P$ be a rooted tree with root vertex labeled $0$. Let $\{A_1,\ldots,A_m\}$ denote the set of isomorphism classes of $\{D(x)\colon x\lessdot 0\}$, and let $i_k$ denote the number of subtrees in $\{D(x)\colon x\lessdot 0\}$ in the isomorphism class $A_k$. For $A_k \in \{A_1,\ldots,A_m\}$, denote $G_k = \Aut(A_k)$. Then, 
\begin{equation}
\label{eq:level_expansion}
\Aut(P) \cong (G_1 \wr S_{i_1}) \times (G_2 \wr S_{i_2}) \times \cdots \times (G_m\wr S_{i_m})
\end{equation}
In particular, $\Aut(P)$ can be expressed as a sequence of direct products and wreath products of symmetric groups.
\end{prop}

\begin{proof}
We proceed by induction on the rank of $P$.  It is clear that if $P$ is rank 0, then $\Aut(P)$ is trivial.  If the rank of $P$ is greater than 0, label the vertices of $P$ by $\{0,1,\ldots, s\}$ such that the root is labeled $0$ and the vertices just below the root are labeled $1, \ldots, k$. Let $A_1,\ldots, A_m$ denote the distinct isomorphism classes of trees in the set $\{D(1),\ldots, D(k)\}$. For $A_k \in \{A_1,\ldots,A_m\}$, denote $G_k = \Aut(A_k)$. 
Let $T_j = \{t\colon t\lessdot 0,D(t) \cong A_j\}$. Then, letting $Q_j$ be the subtree of $P$ whose elements lie in the set $\{0\} \cup (\cup_{t \in T_j} D(t))$, we have that $\Aut(Q_j) \cong G_j \wr S_{i_j}$, because after choosing a permutation of the elements of $T_j$, we are free to choose any element of $G_j$ to permute each $D(t),t \in T_j$. If $t_1 \lessdot 0,t_2 \lessdot 0$, and $g \cdot t_1 = t_2$, then it must be that $g \cdot D(t_1) = D(t_2)$. Therefore, $\Aut(P)$ must permute these isomorphism classes of trees, and the full automorphism groups is simply the direct product, 
\begin{equation}
\Aut(P) \cong (G_1 \wr S_{i_1}) \times (G_2 \wr S_{i_2}) \times \cdots \times (G_m\wr S_{i_m}),
\end{equation}

Since each $G_j$ is a sequence of direct products and wreath products with symmetric groups by the inductive assumption, it follows from ~\eqref{eq:level_expansion} that so is $\Aut(P)$.
\end{proof}

\begin{eg}
Let $P_1$ be the rooted tree in Figure \ref{fig:tree1} and $P_2$ be the rooted tree in Figure \ref{fig:tree2}.  Proposition \ref{prop:automorphism_trees} says that $\Aut (P_1) \cong (S_2 \wr S_2)\wr S_2;$ and $\Aut(P_2) \cong (S_2 \wr S_2) \times (S_3 \wr S_2) $. 
\end{eg}

\begin{cor}
\label{cor:tree_cct}
For $P$ a rooted tree, $\Aut(P)$ is CCT.
\end{cor}
\begin{proof}
Let the rank of $P$ be $n$. By Proposition ~\ref{prop:wreath_preservation}, wreath products with symmetric groups preserve the CCT property, and by Proposition ~\ref{prop:direct_product_preservation} the direct product of two CCT groups is again CCT. Therefore, by the proof of Proposition ~\ref{prop:automorphism_trees}, the group $\Aut(P)$ can be built up by repeating this pair of observations $n-1$ times.
\end{proof}

\sssec{Automorphisms of Polytopes}
\label{sssec:polytopes}

As another class of CCT actions, we describe several linear automorphism groups of polytopes whose induced actions on $B_n$ are CCT. In particular, we prove that the linear automorphism groups of simplices and octahedra are CCT.
Once we prove Proposition ~\ref{prop:cover_transitive_building_blocks} in Subsection \ref{sssec:dihedral}, we will also see that the action of the dihedral group on a regular $n$-gon is CCT for $n = p,2p$.  Since the dihedral group is the group of all linear automorphisms of the regular $n$-gon, this action gives another example of the linear automorphism group of a polytope being CCT.

\begin{defn}
Let $M$ be a polytope with a particular embedding in $\mathbb R^n$. The {\it group of linear automorphisms of M} is the subgroup of $\operatorname{GL}_n$ whose elements are $\{g \in \operatorname{GL}_n\colon g \cdot M = M\}$.
\end{defn}

First we look at linear automorphisms of simplices. Let $G$ be the group of linear automorphisms of the $(n-1)$-simplex whose vertices lie at the standard basis vectors in $\mathbb R^n$. The action of $G$ on the $(n-1)$-simplex induces an action on $[n]$, given by identifying $[n]$ with the $n$ vertices of the $(n-1)$-simplex. Hence, it induces an action on $B_n$.

\begin{eg}
The induced action of the group of linear automorphisms of the $(n-1)$-simplex on $B_n$ is CCT. To see this, observe that the group of linear automorphisms in this case induces the usual action of $S_n$ on $B_n$, because any permutation matrix defines a linear map on $\mathbb R^n$. However, we know the action of $S_n$ on $B_n$ is CCT from Example ~\ref{eg:trivial_edgequot}.
\end{eg}

Next we look at linear automorphisms of octahedra. Let $G$ be the group of linear automorphisms of the n-octahedron whose vertices are located at $\pm e_i$, where $e_1,\ldots e_n$ are the standard basis vectors of $\mathbb R^n$. Then the action of $G$ on the octahedron induces an action of $G$ on the $2n$ vertices of the octahedron, and hence on $B_{2n}$.

\begin{prop}
The induced action of the group of linear automorphisms of the $n$-octahedron on $B_{2n}$ is CCT.
\end{prop}
\begin{proof}
It is simple to see that the group of linear automorphisms of the n-octahedron is the hyperoctahedral group, since it is generated by the permutation matrices together with the matrix $A$, where $A_{1,1} = -1,A_{i,i} = 1$, and $A_{j,k} = 0$ for $i \neq 1, j \neq k$.\footnote{The hyperoctahedral group is commonly denoted by $B_n$, since it is the type $B$ Coxeter group. We do not use this notation here to avoid confusing it with the boolean algebra.} It is well known that the hyperoctahedral group can be written as $S_2 \wr S_n$. Then, by Proposition ~\ref{prop:wreath_preservation}, it follows that $S_2 \wr S_n$ is CCT.
\end{proof}

\begin{rmk}
Let us give a brief recap of which linear automorphisms of polytopes are known to induce actions on $B_n$ which are CCT. First, by the above lemmas, the induced action is CCT for octahedra and simplices. By Proposition  \ref{prop:cover_transitive_building_blocks} and Remark \ref{rem:iff_dihedral_cct}, the linear automorphism group of an $n$-gon induces a CCT action on $B_n$ if and only if $n\in \{1,p,2p\}$ for $p$ a prime. Additionally, using python code, we have verified that automorphisms of the 3-cube with vertices at $(\pm 1,\pm 1,\pm 1)$ induces a CCT action. It is still unknown whether the linear automorphism groups of $n$-cubes are CCT for $n> 3$, and also whether the remaining five exceptional regular polytopes (namely the dodecahedron and icosahedron in $\mathbb R^3$, and the $24$-cell, $120$-cell, and $600$-cell polytopes in $\mathbb R^4$) induce CCT actions. These questions are repeated in Question \ref{question:cube_cct} and Question \ref{question:exceptional_polytopes_cct}.

\end{rmk}

\sssec{CCT Actions of $\mathbb Z_2^k$}

In this subsection, we show that any embedding of $\mathbb Z_2^k$ into $S_n$ defines an action on $B_n$ which is CCT.  This implies that every action of $\mathbb Z_2^k$ on $B_n$ is CCT. However, it turns out that this is the only class of groups for which every action on $B_n$ is CCT.

\begin{prop}
\label{prop:regular_action_CCT}
Recall that $G$ is an elementary abelian $2$-group if $G \cong (\mathbb Z/2\mathbb Z)^k$ for some $k \in \mathbb N$.
\begin{enumerate}
  \item For any $n \in \mathbb N$, and $G$ an elementary abelian 2-group, every $G$-action $\phi:G \times B_n \rightarrow B_n$ is CCT.
  \item For every finite group $G$ which is not an elementary abelian 2-group, there exists at least one $G$-action which is not CCT, namely the action of $G$ on $B_n$ induced by the left-regular action of $G$ on itself, where $n = |G|$.
\end{enumerate}
\end{prop}

\begin{proof}
First, we show $(1)$ holds. Let $x,y,z\in B_n$ such that $x \lessdot z, y \lessdot z$, and  $x = gy$ for some $g\in G$.  Since $x\ne y$ we have $z=x\cup y$.  Furthermore, since every element in $\mathbb Z_2^k$ has order 2 we have that $gy = g^2x = x$ and thus $gz = gx\cup gy = y\cup x = z$. Hence $g\in\Stab(z)$ and thus $\phi$ is CCT.

Next, we show $(2)$ holds. First, let us show $G \cong \mathbb Z_2^k \iff \forall g \in G,g^2 = e$. The forward implication is obvious. To see the converse, first note that if $\forall g \in G$, $g^2 = e$, then $G$ is abelian because $aba^{-1}b^{-1} = abba = a^2 = e$.  Then, $G$ is an abelian group, all of whose elements have order two. The structure theorem of finite abelian groups tells us $G \cong \mathbb Z_2^k$.

Suppose $G \not \cong \mathbb Z_2^k$. Then there exists $g \in G$ such that $g^2 \neq e$. Clearly $\{e\}\lessdot \{e \cup g\},\{g\} \lessdot \{e \cup g\}$, and $\{g\} \in G\{e\}$. So in order to show that the induced action $\phi\colon G \times B_n \rightarrow B_n$ is not CCT, it suffices to show there is no $h \in G$ such that $h \in \Stab(\{e \cup g\})$ and $h\cdot \{e\} =\{g\}$.  If $h \in \Stab(\{e \cup g\})$ and $h \cdot \{e\} = \{g\}$, then $h = g$, and it follows that $g^2 = e$.  Thus there does not exist such an $h$, and the action induced by left multiplication is not CCT.
\end{proof}

\sssec{The proof of Proposition~\ref{prop:cover_transitive_building_blocks}}
\label{sssec:dihedral}\indent 

Let $x, y \in B_n$, and let $G$ act on $B_n$. 
\building*

\begin{proof}
We have already seen in Example \ref{eg:trivial_edgequot} that $(1)$ holds trivially. We prove part $(2)$. The proof of part $(3)$ is similar. 

Let $x$, $y$ be elements of $B_n$ such that $x$ is covered by $y$.  We wish to prove that given $\sigma \in D_{2p}$ such that $\sigma x \lessdot y$, there exists some $\tau \in D_{2p}$ such that $\tau x = \sigma x$ and $\tau y = y$. Of course, we may assume $\sigma x \neq x$, as we could then take $\tau = \id$.

The action of $D_{2p}$ on $B_p$ is induced by the action of $D_{2p}$ on $[p]$, where $[p]$ is identified with vertices of the regular $p$-gon. Note that any element in $D_{2p}$ is either some reflection $r$ by one of the lines of symmetry of the polygon or some rotation $\sigma_0^{d}$, where $\sigma_0$ is the generator $\sigma_0 = (12 \cdots p)$ and $d$ is some integer. Hence we only need to show the claim when $\sigma = r$ or $\sigma = \sigma_0^d$. It is clear that the claim holds for $\sigma = r$: if $x \lessdot y$ and $r \cdot x \lessdot y$, then $r \cdot y = y$, since $r$ is of order 2, and we are assuming $\sigma x \neq x$. Now suppose $\sigma_0^d \cdot x \lessdot y$ for some $(x,y) \in \mathcal E(B_p)_i$. Then $(x, y)$ is of the form $x = \{s, s+d, ..., s+(i-1)\cdot d\}$ for some starting point $s \in [n]$ and $y = \{s, s+d, ..., s+(i-1)d, s + i \cdot d\}$. Now let $r_0$ be the reflection given by $x \mapsto (2s+ i \cdot d)-x$ for all $x \in [n]$, reducing mod $n$ whenever necessary. Then $r_0 x = \sigma_0^d x$ and $r_0$ fixes $y$ by construction.  Therefore the action of $D_{2p}$ on $B_p$ is CCT. 
\end{proof}

\begin{rmk}
\label{rem:iff_dihedral_cct}
We claim that if $n \neq p,n \neq 2p$, and $n >8$ for any prime $p$, then the action of $D_{2n}$ on $B_n$ is not CCT. To see this, we give an example of a non-CCT pair. Assume $n \ne p, 2 p$.  Then $n = mk$ for some $m \ge k \ge 3$. Let us consider elements $x, y, z$, where $z = \{1, m+1, 2m+1, ..., (k-1)m+1, 2, m+2\}$, $x = z \backslash \{m+2\}$, and $y = z \backslash \{2\}$. We immediately have that $x, y \lessdot z$, and $x \in D_{2n} y$ since $x$ is sent to $y$ by the permutation $(12\cdots n)^m \in D_{2n}$. It is also clear from the asymmetry of the element $z$ that there is no $g \in D_{2n}$ translating $x$ to $y$ while fixing $z$. Therefore the action of $D_{2n}$ on $B_n$ as described is CCT if and only if $n =p$ or $n = 2p$ for some prime $p$.

Hence a complete list of $n$ for which $D_{2n}$ is $CCT$ is given by $n = p,n = 2p, n = 1$, and $n = 8$, where $p$ varies over all primes.
\end{rmk}

\begin{rmk}
There are several other results related to $\mathcal E(C_n)$ and $\mathcal E(D_{2n})$, where $C_n = \mathbb Z/n\mathbb Z$, which are proven in Section 7 of the REU report. Notably,
\begin{enumerate}
	\item For $G$ any group of order $n$ acting transitively on $[n]$, the induced action of $G$ on $B_n$ defines a quotient poset with $|(\mathcal E(B_n)/G)_i| = \binom {n-1} i$.
	\item For all $n$, $\mathcal E(B_n/C_n)$ is symmetric and unimodal.
	\item For all $n$, $\mathcal E(B_n/D_{2n})$ is symmetric and unimodal.
\end{enumerate}
\end{rmk}

%%%%%%%%%%%%%%%%% Pak and Panova %%%%%%%%%%%%%%%%%%
%%%%%%%%%%%%%%%%%%%%%%%%%%%%%%%%%%%%%%%%%%%%%
\section{A Unimodality Result}\label{sec:wreath_product}

Using Corollary ~\ref{cor:tree_cct}, we prove a result related to \cite[Theorem 1.1]{pak} of Pak and Panova. We construct a certain sequence which is not only unimodal, but can even be exhibited as the ranks of a Peck poset. This construction gives an alternate proof of \cite[Theorem 1.1]{pak} in the case that $r = 1$.

For this section, fix $\ell,m$ with $n = \ell \cdot m$, and fix $G = S_m \wr S_\ell$. Let $S_m$ act on $B_m$ by the permutation representation, and then let $G$ act on $B_{m}^\ell\cong B_{m \cdot \ell}$ by the action defined in Definition \ref{defn:wreath_action}.

\ssec{Restatement of the Unimodality Result}
We first review the necessary definitions and then state \cite[Theorem 1.1]{pak}:

A {\it partition} $\lambda$ of $n$, denoted by $\lambda \vdash n$, is a sequence of numbers $\lambda = (\lambda_1,\ldots, \lambda_k)$ such that $\lambda_1 \geq \lambda_2 \geq \cdots \geq \lambda_k$ and $\sum_{i=1}^k \lambda_i = n$. Let $P_n(\ell,m)$ denote the set of partitions $\lambda = (\lambda_1,\ldots, \lambda_k) \vdash n$, such that $\lambda_1 \leq m$ and $k \leq \ell$. That is, $P_n(\ell,m)$ is the set of partitions which fit inside an $\ell \times m$ rectangle.

For $\lambda$ a partition, let $\nu(\lambda)$ be the number of distinct nonzero part sizes of $\lambda$.  Let $p_k(\ell,m,r) = \sum_{\lambda \in P_k(\ell,m)} \binom{\nu(\lambda)}{r}$, as defined in \cite[Section 1]{pak}.

A {\it Young diagram} is a finite collection of boxes arranged
so that their rows are left-justified and their row lengths are weakly
increasing from top to bottom.

\begin{thm}
\label{thm:pak_thm}
\cite[Theorem 1.1]{pak}
The sequence $p_r(\ell,m,r), p_{r+1}(\ell,m,r),\ldots, p_{\ell\cdot m}(\ell,m,r)$ is unimodal and symmetric.
\end{thm}

\ssec{A Proof of Theorem \ref{thm:pak_thm} for $r = 1$}
Now that we have stated Pak and Panova's Theorem, we give an alternative proof of Theorem \ref{thm:pak_thm} in the case of $r=1$. In fact, we do better, by realizing the $p_i(\ell,m,1)$ as ranks of a Peck poset.

\begin{prop}
\label{prop:rank_gen_fn_wreath_1}
There is an equality $|(\mathcal E(B_n)/G)_i| = |\mathcal E(B_n/G)_i| = p_{1+i}(\ell,m,1)$. In particular, Theorem \ref{thm:pak_thm} holds in the case $r = 1$.
\end{prop}

\begin{proof}

First, observe that $S_m \wr S_\ell$ can be described as the automorphism group of a rooted tree of rank 2 with $\ell$ elements at rank $1$ and $m\cdot \ell$ elements at rank 2, such that each element at rank $1$ is above $m$ elements at rank 2. Then it follows from Corollary \ref{cor:tree_cct} that the action of $G$ on $B_{m \cdot \ell}$ is CCT and hence $\mathcal E(B_n/G)$ is Peck.

Next, note that each equivalence class in $B_n/G$ has a unique representative which is a Young diagram.  Here, we think of $B_n = B_{\ell \cdot m}$ as subsets of the $\ell \cdot m$ boxes in an $\ell \times m$ rectangle. The correspondence is then given by sending an equivalence class to the representative which is left-justified and bottom-justified.
 For a complete proof, see, for example, Lemma 5.11 of the REU report.

Now let $Gx$ and $Gy$ be two $G$ orbits with $\bar x$ the Young diagram corresponding to $x$ and $\bar y$ the Young diagram corresponding to $y$. Suppose $Gx \lessdot Gy$. Then $\bar x$ must be a subdiagram of $\bar y$ with a single box removed. Since $\bar x$ and $\bar y$ are both Young diagrams, the removed box must be one of the corners of $\bar y$. Observe that the number of corners of a partition is precisely the number of distinct part sizes, and so $|\{Gx\colon Gx \lessdot Gy\}| = \nu(\bar y)$. Thus,
\begin{align*}
  |\mathcal E(B_n/G)_i| &= \sum_{(Gx,Gy) \in \mathcal E(B_n/G)_i} 1 
  \\
  &= \sum_{Gy \in (B_n/G)_{i+1}} \left(\sum_{Gx \lessdot Gy}^{} 1 \right) 
  \\
  &= \sum_{Gy \in (B_n/G)_{i+1}}\nu(\bar y) 
  \\
  &=  \sum_{\lambda \in P_{i+1}(\ell,m)} \nu(\lambda).
\end{align*}

Therefore $|(\mathcal E(B_n)/G)_i| = |\mathcal E(B_n/G)_i| = p_{1+i}(\ell,m,1)$. Since $\mathcal E(B_n)/G$ is Peck,
\linebreak
$p_1(\ell,m,1), p_{2}(\ell,m,1),\ldots, p_{\ell\cdot m}(\ell,m,1)$ is unimodal and symmetric, and hence Theorem \ref{thm:pak_thm} holds in the case $r = 1$.
\end{proof}

%%%%%%%%%%%%%%%%% Final Remarks %%%%%%%%%%%%%%%%%%
%%%%%%%%%%%%%%%%%%%%%%%%%%%%%%%%%%%%%%%%%%%%%
\section{Final Remarks}

In this section, we discuss several related results and list further questions. 

\begin{defn}
Let $B_n(q)$, the {\it q-analog of the boolean algebra}, be the graded poset whose elements are linear subspaces $V \subset \mathbb F_q^n$ with $V \leq W$ if $V \subset W$.
\end{defn}

A natural extension of Conjecture ~\ref{conj:F_of_BnG_Peck} would be an analogous result for q-analogs. We suspect the method used in Section 8 of the REU report to prove $\mathcal E(B_n)$ is unitary Peck may solve Question ~\ref{question:unitary_peck_q_edge}.

\begin{que}
\label{question:unitary_peck_q_edge}
Is $\mathcal E(B_n(q))$ unitary Peck?
\end{que}

Let $G$ be a group acting on $B_n$.  If the answer to Question ~\ref{question:unitary_peck_q_edge} is affirmative, it immediately follows that $\mathcal E(B_n(q))/G$ is Peck.  Furthermore, if the action of $G$ is CCT, then this implies that $\mathcal E(B_n(q)/G)$ is Peck. Hence we pose the following question.

\begin{que}
For $G$ a group with a CCT action on $B_n(q)$, is $\mathcal E(B_n(q)/G)$ Peck?
\end{que}

More generally, we wonder if the q-analog of Conjecture ~\ref{conj:F_of_BnG_Peck} holds.

\begin{que}
For $G$ a group acting on $B_n(q)$, is $\mathcal E(B_n(q)/G)$ Peck? If not, is $\mathcal E(B_n(q)/G)$ rank-unimodal?
\end{que}

We remark that Stanley was able to answer the rank unimodality question for many cases using hard Lefschetz type theorems, by realizing the ranks of the poset as ranks of cohomology groups coming from algebraic geometry \cite{stanley:combinatorial-applications-of-the-hard-lefschetz-theorem}.
\begin{que}
Can we realize the ranks of $\mathcal E(B_n(q)/G)$, or even the edge poset itself, geometrically? 
\end{que}

We found several additional interesting examples of CCT actions. Once such action is the linear automorphism of the $n$-cube. Using python code we wrote, we found that for $n \leq 3$, the linear automorphisms of the $n$-cube induces a CCT action on $B_{2^n}$. We wonder if this generalizes.

\begin{que}
\label{question:cube_cct}
Does the group of linear automorphisms of an $n$-cube in $\mathbb R^n$ whose vertices lie at $(\pm 1, \ldots, \pm 1)$ induce a CCT action on $B^{2^n}$?
\end{que}

There is also the question of which regular polytopes induce CCT actions. We have shown that the $n$-octahedron (respectively the $n$-simplex) induces a CCT action on $B_{2n}$ (respectively $B_{n+1}$) in Subsection~\ref{sssec:polytopes}. We also checked using python code that the $n$-cube induces a CCT action on $B^{2^n}$ for $n \leq 3$. We wonder whether the induced action is CCT for the five exceptional regular polytopes, which are the only regular polytopes other than simplices, octahedra, and cubes.

\begin{que}
\label{question:exceptional_polytopes_cct}
Do the groups of linear automorphisms of the five exceptional regular polytopes (namely the dodecahedron and icosahedron in $\mathbb R^3$ and the $24$-cell, $120$-cell, and $600$-cell polytopes in $\mathbb R^4$) induce CCT actions?
\end{que}

We found using python code we wrote that the group of invertible linear maps on $\mathbb F_2^3$ acting on the the seven nonzero points of $\mathbb F_2^3$ induces an action on $B_7$ which is CCT. We wonder if this generalizes to other groups of invertible linear maps on finite fields.

\begin{que}
Is the action of $\operatorname{GL}_n(\mathbb F_q)$ on $B_{q^n-1}$ (induced by the action of $\operatorname{GL}_n(\mathbb F_q)$ on $(\mathbb F_q^n)^\times$) CCT? What about the action of $\operatorname{PGL}_n(\mathbb F_q)$ on $B_n(q)?$ If not, what about the action of $\operatorname{PGL}_n(\mathbb F_2)$ on $B_n(2)?$ 
\end{que}

\begin{acknowledgements}
This research was carried out in the 2014 combinatorics REU program at the University of Minnesota, Twin Cities and was supported by RTG grant NSF/DMS-1148634.
We would like to thank our mentor Victor Reiner for his consistent help and guidance throughout the project and our TA Elise DelMas for her helpful feedback on the paper. We would also like to thank Ka Yu Tam for helpful comments.  In addition, we thank the math department of University of Minnesota, Twin Cities, for its hospitality and Gregg Musiker for organizing the program.
\end{acknowledgements}

Conflict of Interest: The authors declare that they have no conflict of interest.

\medskip 

\nocite{*}
\bibliographystyle{plain}
\bibliography{References}

\begin{thebibliography}{1}

\bibitem{babson}
Eric Babson and Dmitry~N. Kozlov.
\newblock Group actions on posets.
\newblock {\em J. Algebra}, 285:439--450, 2005.

\bibitem{greene}
Curtis Greene and Daniel~J. Kleitman.
\newblock Proof techniques in the theory of finite sets.
\newblock In {\em Studies in combinatorics (MAA Stud. Math.)}, volume~17, pages
  22--79. Math. Assoc. America, Washington, D.C., 1978.

\bibitem{pak}
Igor Pak and Greta Panova.
\newblock Unimodality via {K}ronecker products.
\newblock {\em J. Algebraic Combin.}, 40:1103--1120, 2014.

\bibitem{proctor}
Robert~A. Proctor.
\newblock Representations of {$\mathfrak s \mathfrak l(2,\mathbb C)$} on posets
  and the {S}perner property.
\newblock {\em SIAM J. on Algebraic and Discrete Methods}, 3:275--280, 1982.

\bibitem{weyl_groups_stanley}
Richard~P. Stanley.
\newblock Weyl groups, the hard {L}efschetz theorem, and the {S}perner
  property.
\newblock {\em SIAM J. on Algebraic and Discrete Methods}, 1:168--184, 1980.

\bibitem{stanley:combinatorial-applications-of-the-hard-lefschetz-theorem}
Richard~P. Stanley.
\newblock Combinatorial applications of the hard {L}efschetz theorem.
\newblock In {\em Proceedings of the {I}nternational {C}ongress of
  {M}athematicians, {V}ol.\ 1, 2 ({W}arsaw, 1983)}, pages 447--453. PWN,
  Warsaw, 1984.

\bibitem{quotients_stanley}
Richard~P. Stanley.
\newblock Quotients of {P}eck posets.
\newblock {\em Order}, 1:29--34, 1984.

\bibitem{enumerative_comb}
Richard~P. Stanley.
\newblock {\em Enumerative Combinatorics}, volume~1.
\newblock Cambridge University Press, 2011.

\bibitem{stanley_alg_comb}
Richard~P. Stanley.
\newblock {\em Algebraic Combinatorics}.
\newblock Undergraduate Texts in Mathematics. Springer, 2013.

\end{thebibliography}

\end{document}